\documentclass[a4paper,11pt]{amsart}
\usepackage{amssymb,amsmath,amsthm}
\usepackage{fullpage, subcaption}
\usepackage{color, mathtools, multicol}
\definecolor{darkred}{rgb}{0.6,0.2,0.2}
\usepackage[colorlinks=true,linkcolor=blue, citecolor=darkred]{hyperref}
\usepackage{eucal, beton, lmodern}
\usepackage{float}

\usepackage{xcolor, tcolorbox}
\usepackage{stackrel, enumerate}
\usepackage{graphicx, bbm}

\definecolor{darkblue}{rgb}{0.2,0.2,0.6}

\newcommand\nb{\nabla}
\newcommand{\beq}{\begin{equation} \begin{split}}
\newcommand{\eeq}{\end{split} \end{equation}}
\newcommand\Sg{\Sigma}
\newcommand\Omg{\Omega}

\makeatletter
\def\section{\@startsection{section}{1}\z@{.9\linespacing\@plus\linespacing}%
	{.7\linespacing} {\fontsize{13}{14}\selectfont\bfseries\centering}}
\def\paragraph{\@startsection{paragraph}{4}%
	\z@{0.3em}{-.5em}%
	{$\bullet$ \ \normalfont\itshape}}

\@addtoreset{equation}{section}
\makeatother

\renewcommand\and{\qquad\text{and}\qquad}

\newcommand\sm{\setminus}

\newcommand\dl{\delta}

\newcommand{\comm}[1]{}

\def\sfS{\mathsf{S}}
\def\sfS{\mathsf{S}}

\def\bm1{\mathbbm{1}}
\def\G{\Gamma}

\def\s{\sigma}

\def\p{\partial}

\def\omg{\omega}

\def\Re{{\rm Re}\,}

\def\arr{\rightarrow}

\newcommand{\sfJ}{\mathsf{J}}

\def\lm{\lambda}

\def\s{\sigma}

\def\sess{\sigma_{\rm ess}}

\def\ii{{\mathsf{i}}}
\def\p{\partial}

\def\kp{\kappa}

\def\sfP{\mathsf{P}}
\def\dd{{\,\mathrm{d}}}

\def\omg{\omega}

\def\sfS{\mathsf{S}}

\def\sfP{\mathsf{P}}

\newcounter{counter_a}
\newenvironment{myenum}{\begin{list}{{\rm(\roman{counter_a})}}%
{\usecounter{counter_a}
\setlength{\itemsep}{1.ex}\setlength{\topsep}{0.8ex}
\setlength{\leftmargin}{5ex}\setlength{\labelwidth}{5ex}}}{\end{list}}

\usepackage[latin1]{inputenc}
\usepackage[T1]{fontenc}

\newcommand{\eg}{{\it e.g.}\,}
\newcommand{\ie}{{\it i.e.}\,}
\newcommand{\cf}{{\it cf.}\,}

\numberwithin{figure}{section}
\numberwithin{equation}{section}
\theoremstyle{plain}
\newtheorem*{thm*}{Theorem}
\newtheorem{thm}{Theorem}[section]
\newtheorem{hyp}[thm]{Hypothesis}
\newtheorem{lem}[thm]{Lemma}
\newtheorem{prop}[thm]{Proposition}

\theoremstyle{remark}
\newtheorem{remark}[thm]{Remark}
\theoremstyle{plain}

\newcommand{\supp}{\mathrm{supp}\,}
\newcommand{\beu}{\begin{equation*}}
\newcommand{\eeu}{\end{equation*}}
\newcommand{\besu}{\begin{equation*}
\begin{aligned}}
\newcommand{\eesu}{\end{aligned}
\end{equation*}}
\newcommand{\bes}{\begin{equation}
\begin{aligned}}
\newcommand{\ees}{\end{aligned}
\end{equation}}

\newcommand\cA{\mathcal A}
\newcommand\cB{\mathcal B}
\newcommand\cD{\mathcal D}

\newcommand\cH{\mathcal H}

\newcommand\CC{\mathbb C}
\newcommand\NN{\mathbb N}
\newcommand\RR{\mathbb R}

\newcommand\fra{\mathfrak a}
\newcommand\frq{\mathfrak q}

\newcommand\frp{\mathfrak p}

\newcommand\eps{\varepsilon}

\newcommand\ov{\overline}
\newcommand\wt{\widetilde}

\newcommand\void[1]{}

\def\ov{\overline}
\def\eps{\varepsilon}

\def\frb{{\mathfrak b}}

\def\frt{{\mathfrak t}}

      \def\dC{{\mathbb C}}

   \def\dN{{\mathbb N}}   
      \def\dR{{\mathbb R}}

\def\sfD{{\mathsf D}}      
      
\def\sfJ{{\mathsf J}}      
      
\def\sfP{{\mathsf P}}      
\def\sfS{{\mathsf S}}   \def\sfT{{\mathsf T}}

\def\cA{{\mathcal A}}   \def\cB{{\mathcal B}}   
\def\cD{{\mathcal D}}      
   \def\cH{{\mathcal H}}

      \def\cU{{\mathcal U}}

\newcommand{\dom}{\mathrm{dom}\,}

\newcommand\frd{\mathfrak{d}}
\newcommand\fri{\mathfrak{i}}

\usepackage[left=2.7cm, right=2.7cm, marginparwidth=2cm, textheight = 24cm ]{geometry}
\usepackage[normalem]{ulem}
\definecolor{DarkGreen}{rgb}{0,0.5,0.1}

\definecolor{DarkBlue}{rgb}{0,0.1,0.5}

\newcommand\soutD{\bgroup\markoverwith
	{\textcolor{DarkGreen}{\rule[.5ex]{2pt}{1pt}}}\ULon}
\newcommand\soutP{\bgroup\markoverwith
	{\textcolor{blue}{\rule[.5ex]{2pt}{1pt}}}\ULon}
\newcommand{\Hm}[1]{\leavevmode{\marginpar{\tiny%
			$\hbox to 0mm{\hspace*{-0.5mm}$\leftarrow$\hss}%
			\vcenter{\vrule depth 0.1mm height 0.1mm width \the\marginparwidth}%
			\hbox to
			0mm{\hss$\rightarrow$\hspace*{-0.5mm}}$\\
			\relax\raggedright #1}}}

\newcommand{\R}{\mathbb{R}}

\newtheorem{theorem}{Theorem}[section]

\newtheorem{lemma}[theorem]{Lemma}

\newcounter{intro}

\newtheorem{maintheorem}[intro]{Theorem}

\theoremstyle{remark}
\newtheorem{ex}[thm]{Example}

\begin{document}
\title{On the discrete spectrum of Dirac operators with Lorentz-scalar \boldmath{$\dl$}-shell interactions supported on unbounded curves}

\author[M. Holzmann]{Markus Holzmann}
\address{(M.~Holzmann)
	Institut f\"{u}r Angewandte Mathematik,
Technische Universit\"{a}t Graz,
Steyrergasse 30, A 8010 Graz, Austria
}
\email{holzmann@math.tugraz.at}
\author[V.~Lotoreichik]{Vladimir Lotoreichik}
\address{(V.~Lotoreichik)
	Department of Mathematics, Faculty of Nuclear Sciences and Physical
Engineering,
 Czech Technical University in Prague, Trojanova 13, 120 00, Prague,
Czech Republic 
}
\email{vladimir.lotoreichik@gmail.com}

\author[M. Vogel]{Marco Vogel}
\address{(M.~Vogel)
	Fakult\"{a}t f\"{u}r Mathematik,
Technische Universit\"{a}t Dortmund,
Vogelpothsweg 87, 44227 Dortmund, Germany
}
\email{marco.vogel@tu-dortmund.de}



\begin{abstract}
	We consider the massive Dirac operator (with positive mass) in the plane with an attractive Lorentz-scalar $\delta$-shell interaction of strength $\tau\in(-\infty,0)\sm\{-2\}$ supported on a $C^\infty$-smooth curve $\Sg\subset\dR^2$ being a local deformation of the broken line.
	This singular interaction is defined by imposing a suitable transmission condition on the curve $\Sg$ in the operator domain.
	Such a Dirac operator is self-adjoint and has a gap in the essential spectrum, whose size is explicit and depends on the mass and the interaction strength. We show that the number of discrete eigenvalues in the gap is finite. Under the assumption that one of the domains bounded by $\Sg$ is convex, we prove that the corresponding Dirac operator has a non-empty discrete spectrum, provided that $\tau$ is either sufficiently small or sufficiently large in absolute value. The result holds for any perturbation of a broken line of any opening angle as described above, and this discrete spectrum is induced by the geometry, since for the same type of a singular interaction supported on the straight line the discrete spectrum is empty.
\end{abstract}

\maketitle

\section{Introduction}
Certain geometric deformations of a waveguide-type domain or of the support of a singular interaction can induce discrete eigenvalues for the respective operator. The corresponding discrete spectrum is referred to as geometrically induced and has been studied extensively over the last three decades; see the monograph~\cite{EK15} and the references therein. 

The study of this phenomenon was initiated in~\cite{DE95, ES89}, where the existence of bound states in a bent quantum waveguide was demonstrated; in other words, the existence of the discrete spectrum for the Dirichlet Laplacian on a tubular neighbourhood of a bent curve was proved. The analysis was extended in~\cite{CEK04, DEK01} to quantum layers, that is for Dirichlet Laplacians on tubular neighbourhoods of surfaces in $\dR^3$, where existence of the discrete spectrum was established under certain additional assumptions on the curvatures. 

The effect of geometrically induced bound states manifests also in the setting of singular interactions. The motivation for this stems from the fact that via approximation procedures the existence of geometrically induced bound states can be transferred from the models with singular potentials to Schr\"odinger operators with strongly localized regular potentials \cite{BEHL17, EI01, EK03, H26}. It was shown in~\cite{EI01} that the two-dimensional Schr\"odinger operator with an attractive $\delta$-interaction supported on an asymptotically straight and sufficiently regular bent curve in the plane has discrete spectrum below the bottom of the essential spectrum. This result was extended in~\cite{EK03} to surfaces in $\dR^3$ under the assumption that the interaction is sufficiently strong.  While the general result in three dimensions is missing, it was demonstrated in~\cite{EKL18} that $\dl$-interactions supported on a sufficiently small local deformation of the plane in $\dR^3$ induce exactly one bound state below the bottom of the essential spectrum and its asymptotics in the small deformation limit was found. Schr\"odinger operators with $\dl$-interactions supported on certain non-local deformations of the plane such as unbounded conical surfaces can have even infinite discrete spectrum~\cite{BEL14, LO16, OP18}. In recent years, the existence of geometrically induced bound states for soft quantum waveguides, which are modelled by Schr\"odinger operators in the full Euclidean space and potentials that are supported in a neighborhood of a reference curve or surface, was investigated in \cite{E20, E22, E22-2, EKL24, ES24, KKK21}.

The understanding of geometrically induced discrete spectrum for Dirac operators is much less complete than for Schr\"odinger operators. Dirac operators are used to describe particles with spin $\frac{1}{2}$ taking effects of the special theory of relativity into account \cite{T92} and in dimension two they play a role in the mathematical study of graphene \cite{NGM04}. The analysis of the discrete spectrum of Dirac operators is challenging, as the essential spectrum consists of two unbounded intervals, making the operator unbounded from above and from below. Existence of discrete spectrum in the gap of the essential spectrum was addressed for two-dimensional Dirac waveguides with infinite mass~\cite{BBKO22} and zig-zag~\cite{EH22} boundary conditions. More recently, the same question was  studied for Dirac operators with singular interactions supported on curves. The present paper
also addresses the latter setting and significantly extends the admissible geometric configurations for which the existence of the discrete spectrum in the gap of the essential spectrum is proved.

Let us discuss in more detail the results of the present paper and also the preceding results on the existence of discrete eigenvalues for Dirac operators with singular interactions supported on curves. 
Let $\Sg\subset\dR^2$ be a $C^\infty$-smooth non-compact curve (without self-intersections) which coincides with the broken line 
\begin{equation}\label{eq:broken}
	\Sg_0 = 
	\{(s\sin\omg,s\cos\omg)\colon s > 0\}\cup
	\{(s\sin\omg,-s\cos\omg)\colon s \ge 0\},\qquad \omg\in\left(0,\tfrac{\pi}{2}\right),
\end{equation}
outside a disk of a sufficiently large radius.  The curve $\Sg$ naturally splits the plane into two unbounded domains $\Omg_\pm\subset\dR^2$
with the convention that $\Omg_+$ is the local deformation of the sector with angle less than $\pi$ at the vertex.
By $\nu = (\nu_1,\nu_2)^\top$ we denote the inner unit normal for $\Omg_+$. We consider the operator $\sfD_{\tau,m}$ in $L^2(\dR^2;\dC^2)$ associated with the formal differential expression
\[
	-\ii\s_1\p_1-\ii\s_2\p_2 + m\s_3 + \tau\s_3\dl_\Sg,
\]
where $\s_1,\s_2, \s_3$ are the $2\times 2$ Pauli matrices defined in~\eqref{eq:Pauli}, $m > 0$ is the mass, $\tau\in\dR\sm\{-2,0,2\}$ is the strength of the Lorentz-scalar $\dl$-interaction, and $\dl_\Sg$ is the $\dl$-distribution supported on $\Sg$. The operator $\sfD_{\tau,m}$ is precisely defined by imposing a suitable transmission condition on $\Sg$ as
\begin{equation}\label{eq:Op}
	\begin{aligned}
		\sfD_{\tau,m} u &:= \big(-\ii\s_1\p_1-\ii\s_2\p_2+m\s_3\big)u \quad \text{in } \mathbb{R}^2 \setminus \Sigma,\\
		\dom\sfD_{\tau,m} &:= \Big\{
		u = u_+\oplus u_-\in H^1(\Omg_+;\dC^2)\oplus H^1(\Omg_-;\dC^2)\colon\\
	&\qquad\qquad  \ii \sigma_3 \big(\sigma_1\nu_1+\sigma_2\nu_2\big)  (u_+|_\Sigma - u_-|_\Sigma) = \frac{\tau}{2} (u_+|_\Sigma + u_-|_\Sigma)  \Big\},
	\end{aligned}	
\end{equation}
where we use the notation $u_\pm = u|_{\Omg_\pm}$.
The investigation of Dirac operators with singular interactions
	supported on surfaces was initiated in~\cite{DES89}, where the support of the interaction was a sphere and a separation of variables ansatz was employed.
Dirac operators with singular potentials supported on general compact surfaces in $\mathbb{R}^3$ were first introduced in \cite{AMV14, AMV15} and further studied in, \eg, \cite{BEHL18,BEHL19, HOP18}, and two-dimensional Dirac operators with
singular potentials supported on smooth compact curves were introduced in \cite{BHOBP, CLMT23}; see also the review paper \cite{BHSS24} for a unified treatment of the two- and three-dimensional case and  further references. Eventually, Dirac operators with singular interactions supported on unbounded curves and surfaces in $\mathbb{R}^2$ and $\mathbb{R}^3$ were investigated in \cite{BEHT25, B22,FHL24, FL23,  R21, R22a}.
In particular, it is known that
the operator $\sfD_{\tau,m}$ is self-adjoint in $L^2(\dR^2;\dC^2)$; see~\cite{BEHT25}.
Moreover, $\sfD_{\tau,m}$ arises as the norm resolvent limit of Dirac operators with scaled regular Lorentz-scalar potentials localized in the vicinity of $\Sg$; \cf~\cite{BHS25, BHS26, CLMT23, MP18} for details on this approximation problem. In particular, these approximation results allow to transfer the main results obtained below for Dirac operators with singular potentials to Dirac operators with regular potentials in a similar way as in \cite[Section~4]{H26}.

For technical reasons we do not treat the confinement case $\tau = \pm 2$, because the operator $\sfD_{\pm 2,m}$ decouples into an orthogonal sum of two Dirac operators on $\Omg_-$ and $\Omg_+$, respectively, with infinite mass boundary conditions~\cite[Proposition~2.2]{BEHT25}; however, in the confinement case no geometrically induced eigenvalues exist, see Remark~\ref{remark_confinement} for details. Moreover, we also exclude $\tau=0$, which corresponds to the free Dirac operator and hence has no discrete eigenvalues.
For $\tau\in(0,\infty)\sm\{2\}$, we have $\sess(\sfD_{\tau,m}) = (-\infty,-m]\cup[m,\infty)$ and the discrete spectrum of $\sfD_{\tau, m}$ is empty; see Proposition~\ref{essentials}\,(i) and (iv) below.
For $\tau\in(-\infty,0)\sm\{-2\}$,
it was shown in~\cite{BEHT25} that the essential spectrum
is given by
\begin{equation}\label{eq:ess}
\sess(\sfD_{\tau,m}) = \left(-\infty,-m\left|\tfrac{\tau^2-4}{\tau^2+4}\right|\right]\cup
\left[m\left|\tfrac{\tau^2-4}{\tau^2+4}\right|,+\infty\right).
\end{equation}
The question of the presence of discrete eigenvalues in the gap of the essential spectrum is subtle. When the interaction support is the straight line, there is no discrete spectrum in the gap~\cite{BHT23}. It was demonstrated in~\cite{FL23} that after dropping the $C^\infty$-smoothness assumption on $\Sg$ and considering the case when $\Sg$ coincides with the broken line $\Sg_0$, the
operator $\sfD_{\tau,m}$ is only symmetric with deficiency indices $(1,1)$.
According to~\cite{FHL24}, the essential spectrum of any of its self-adjoint extensions still coincides with the set in~\eqref{eq:ess}
while the discrete spectrum of any of such self-adjoint extension is non-empty
for any $\tau\in(-\infty,0)\sm\{-2\}$, provided that the angle at the vertex of the broken line is sufficiently small, where the required smallness of the angle depends on the interaction strength $\tau$. Moreover, the number of discrete eigenvalues exceeds any integer as the angle tends to zero. Precise quantitative estimates on the number of discrete eigenvalues were also obtained in~\cite{FHL24}. The analysis was later extended in~\cite{BEHT25} using the same family of trial functions to local deformations of the broken line with a small angle between the rays.

In the present paper, we address the complementary case, when the angle at the vertex of the underlying broken line is not assumed to be sufficiently small. Instead we impose the assumptions on the coupling constant $\tau$. Our main results are collected in the theorem below.
\begin{maintheorem}\label{thm:main}
	Let $m > 0$ and $\tau\in(-\infty,0)\sm\{-2\}$. Let $\Sg\subset\dR^2$ be a $C^\infty$-smooth curve  which is a local deformation of the broken line $\Sg_0$ in~\eqref{eq:broken}.
	Assume that $\Sg$ splits the plane into domains $\Omg_\pm\subset\dR^2$ with the convention that $\Omg_+$ is a local deformation of the sector whose opening angle at the vertex is less than $\pi$. Let the self-adjoint Dirac operator $\sfD_{\tau,m}$ be as in~\eqref{eq:Op}. Then the following statements hold:		
	\begin{myenum}
		\item The discrete spectrum of $\sfD_{\tau,m}$ is at most finite.
		\item Assume, in addition, that $\Omg_+$ is convex. Then there exists $\tau_\star = \tau_\star(\Sg,m) \in(-2,0)$ such that the
		discrete spectrum of $\sfD_{\tau,m}$ is non-empty
		for all $\tau\in(-\infty,\frac{4}{\tau_\star})\cup(\tau_\star,0)$.
	\end{myenum}
\end{maintheorem}
Items (i) and (ii) of Theorem~\ref{thm:main} are restated in Theorems~\ref{finite prop} and~\ref{thm:disc}, respectively.
Both items are proved by applying the min-max principle to the quadratic form associated with the square of the Dirac operator $\sfD_{\tau,m}$; \cf Proposition~\ref{essentials}. Similar formulas for the quadratic form were also derived and played a key role in \cite{ALTR17, FHL24, HOP18, LO18}.

The proof of finiteness of the discrete spectrum relies on splitting the plane into subdomains using Neumann bracketing and the IMS formula. 
This yields essentially a ''partition'' of the operator into operators acting on domains which tile $\RR^2$. After the splitting and a sequence of transforms the finiteness of the discrete spectrum of $\sfD_{\tau,m}$ follows from a Bargmann-type estimate for the one-dimensional Schr\"odinger operator, whose potential depends on $\tau,m$ and the curve $\Sg$.  A similar strategy was successfully employed to show the finiteness of the discrete spectrum, \eg, of 
the magnetic Neumann Laplacian on a half-space~\cite{MT}, of the Dirichlet Laplacian on the broken waveguide~\cite{DLR12}, and of the Robin Laplacian acting on a sector~\cite{khalile}.
We remark that it is the first time that the finiteness of the discrete spectrum of a Dirac operator with  a singular potential supported on an unbounded curve is proved.

The proof of the existence of the discrete spectrum relies on the construction of a suitable trial function using parallel coordinates naturally associated with the curve $\Sg$; \cf Subsection~\ref{section_geometry} for a summary of the necessary geometric notions and related references. Convexity of $\Omg_+$ ensures that the cut-locus in the domain $\Omg_-$ is empty, which is essential in the construction. On the other hand, the cut-locus in $\Omg_+$ is explicit outside a disk of a sufficiently large radius and coincides with the bisector line of $\Sg_0$, which is also used in the construction of the trial function. This construction is largely inspired by the proof of existence of bound states for soft quantum waveguides and layers~\cite{KKK21, KK24}.

\subsection*{Structure of the paper}
Section~\ref{sec:prelim} consists of two subsections and contains preliminary facts on geometry and the definition of the Dirac operator.
In Subsection~\ref{section_geometry} we recall the construction of parallel coordinates in the plane associated with the curve $\Sg$. In Subsection~\ref{ssec:operator}
we define the Dirac operator with Lorentz-scalar $\dl$-shell interaction supported on $\Sg$ and list its basic properties. In particular, we find in this subsection the quadratic form for the square of $\sfD_{\tau,m}$. 
The proof of the finiteness of the discrete spectrum in the gap is provided in Section~\ref{sec:finiteness}. The proof of existence of discrete spectrum in the gap under suitable assumptions on $\Sg$ and $\tau$ is presented in Section~\ref{sec:existence}.
\subsection*{Notations}
The spectrum and the essential spectrum of a self-adjoint operator $\sfT$ in a Hilbert space $\cH$ are denoted by $\s(\sfT)$ and $\sess(\sfT)$, respectively.

For an infinite-dimensional Hilbert space $\mathcal{H}$, $n\in\NN$, and a closed, densely defined, semibounded symmetric sesquilinear form $\frt$ with domain $\dom\frt\subset\mathcal{H}$, let
\begin{equation} \label{min_max_form}
	\lambda_n(\frt):=\inf_{\substack{V\subset\dom(\frt)\\ \mathrm{dim}V=n}}\sup_{\substack{u\in V\\u\neq 0}}\frac{\frt[u]}{\|u\|_\mathcal{H}^2}.
\end{equation}
Moreover, we set for $\lambda\in\RR$
\begin{equation} \label{eigenvalue_counting_form}
	N(\frt,\lambda):=\#\{n\in\NN\colon \lambda_n(\frt) < \lambda\}.
\end{equation}
Whenever $N(\frt,\lambda)$ is finite, it coincides with the number of eigenvalues (taking multiplicities into account) of the operator associated with $\frt$ that are smaller than $\lambda$.

Finally, for an open set $\Omega \subset \mathbb{R}^2$ Sobolev spaces of order $s \in \mathbb{R}$ of $\mathbb{C}^d$-valued functions, $d \in \mathbb{N}$, are denoted by $H^s(\Omega; \mathbb{C}^d)$; \cf \cite{M00}. Likewise, if $\Sigma$ as above is equipped with the Hausdorff measure, that is denoted by $\mathcal{H}$, then Sobolev spaces of order $s \in \mathbb{R}$ of $\mathbb{C}^d$-valued functions, $d \in \mathbb{N}$, are denoted by $H^s(\Sigma; \mathbb{C}^d)$.

\section{Preliminaries}
\label{sec:prelim}

\subsection{Geometric preliminaries} \label{section_geometry}
In this subsection, we introduce a class of unbounded smooth curves in the plane that are local deformations of a broken line and serve as supports of the $\delta$-shell interactions considered below. Under the additional assumption of convexity of one of the domains bounded by such a curve, we will define the parallel coordinates in the plane naturally associated with such a curve and recall some of their known properties, which will be employed in our subsequent analysis of the Dirac operator with a singular interaction.

Let $\Sigma\subset\dR^2$ be a $C^\infty$-smooth curve (without self-intersections), which coincides with a broken line outside a disk of a sufficiently large radius; here by broken line we understand two half-lines emerging from a single point
and constituting an angle of magnitude less than $\pi$. In other words, the curve $\Sigma$ can be viewed as a 'local deformation' of such a broken line.

The curve $\Sg$ splits the plane into two domains $\Omg_+\subset\dR^2$ and $\Omg_- := \dR^2\sm\ov{\Omg_+}$ both naturally being local deformations of sectors. We choose
the convention that $\Omg_+$ is a local deformation of the sector with an angle of magnitude less than $ \pi$. By $\nu = (\nu_1,\nu_2)^\top$, we denote the unit normal vector to the curve $\Sg$ pointing inwards of the domain $\Omg_+$, and by $\gamma=(\gamma_1,\gamma_2)^\top$ we denote an arc-length parametrization of $\Sg$ (\ie $|\gamma'| =1$) which is oriented such that $\nu = (\gamma_2', -\gamma_1')^\top$. It will be convenient for our analysis to specify that $\Sg$ is a local deformation of the fixed broken line
\begin{equation}\label{eq:Sigma0}
	\Sg_0 = 
	\{(s\sin\omg,s\cos\omg)\colon s > 0\}\cup
	\{(s\sin\omg,-s\cos\omg)\colon s \ge 0\},\qquad \omg\in\left(0,\tfrac{\pi}{2}\right),
\end{equation}
and to fix the parametrization of the curve $\Sg$ in such a way
that
$\lim_{s\arr\pm\infty}\gamma_2(s) = \pm \infty$ or, in other words, the parametrization of $\Sg$ is in the clockwise direction. We remark that the opening angle at the vertex of $\Sg_0$ has magnitude $\pi -2\omg$. 
We define the signed curvature $\kp\colon\dR\arr\dR$ of $\Sg$ through the identity
\begin{equation}\label{eq:curvature}
\kp(s) = \gamma_1''(s)\gamma_2'(s)-\gamma_2''(s)\gamma_1'(s).
\end{equation}
With this convention for the curvature in hands, we observe that the Frenet formula 
\begin{equation} \label{Frenet}
  \gamma''(s) = \kp(s)\nu(s)
\end{equation}
holds, that the total curvature is given by
\begin{equation*}
	\int_\Sg\kp\dd \mathcal{H} = 2\omg
\end{equation*}
and that $\kp\ge 0$ if and only if, $\Omg_+$ is convex.
Throughout the paper we will mainly focus on a special subclass of such curves for which the domain $\Omg_+$ is convex. Let us collect all our assumptions on the curve $\Sg$ in a hypothesis. In the following $\cB_R\subset\dR^2$ denotes the disk of radius $R > 0$ centered at the origin.

\begin{hyp}\label{hyp}
	Let  $\Sg\subset\dR^2$ be an unbounded $C^\infty$-smooth curve
	without self-intersections
	which is a local deformation of the broken line $\Sg_0$ defined in~\eqref{eq:Sigma0}; \ie there exists $R >0$ such that $\Sg\sm\cB_R = \Sg_0\sm\cB_R$. 
	
	Let the curve $\Sg$ be parametrized by the arc-length mapping $\gamma\colon\dR\arr\dR^2$ in the clockwise direction and assume that the curve $\Sg$ is such that the  associated signed curvature $\kp$ in~\eqref{eq:curvature} is non-negative or, equivalently, the domain $\Omg_+$ is convex. 
\end{hyp}

\begin{center}
\begin{figure}[h]
\includegraphics[width=4cm, keepaspectratio]{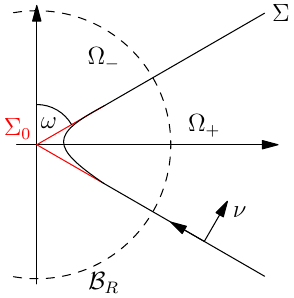}
	\caption{A typical example of our geometric setting: The smooth curve $\Sigma$ coincides with the broken line $\Sigma_0$ (plotted in red) outside of a disk $\cB_R\subset\dR^2$. The curve $\Sigma$ splits $\mathbb{R}^2$ into two domains $\Omega_+$ and $\Omega_-$, where  $\Omg_+$ is a local deformation of the sector with an angle of magnitude less than $ \pi$. The normal vector $\nu$ points inwards of $\Omega_+$, the orientation of the arc-length parametrization $\gamma$ is indicated by an arrow.}
	\end{figure}
\end{center}
Our next goal is to associate parallel coordinates to the curve $\Sg$ satisfying Hypothesis~\ref{hyp}.
The theory of parallel coordinates associated with a curve in the Euclidean space or a manifold is developed in~\cite{H64}.
The reader may consult the monograph~\cite{SST03}
for a modern exposition of the theory.
Below we follow the lines and use the notation of~\cite[Appendix 1]{S01} (see also~\cite[\S 2]{KK24} for the extension to a class of non-compact boundaries). We consider the normal exponential map
\begin{equation}\label{eq:Phi}
\Phi\colon\dR^2\arr\dR^2,\qquad \Phi(s,t) := \gamma(s) + t\nu(s).
\end{equation}
Note also that, by the Frenet formula~\eqref{Frenet}, the Jacobian $\sfJ_\Phi$ of $\Phi$ is given by 
\begin{equation}\label{eq:J}
	\sfJ_\Phi(s,t) = 
	\begin{pmatrix}
	\gamma_1'(s) -t\kappa(s)\gamma'_1(s) & \gamma_2'(s)\\
	\gamma_2'(s) - t\kappa(s)\gamma_2'(s) & -\gamma_1'(s)
	\end{pmatrix}
\end{equation}
and thus its determinant can be computed as
\begin{equation}\label{eq:detJacobian}
\det\sfJ_\Phi(s,t) = -1+\kp(s)t,
\end{equation}
where we used that $|\gamma'| = 1$.
Our aim is to construct parallel coordinates $(s,t)$ based on $\Sg$.
For a curve $\Sg\subset\dR^2$ as in Hypothesis~\ref{hyp} there exists an $a > 0$, such that the restriction $\Phi|_{\R\times(-\infty,a)}$ is a diffeomorphism, where the left endpoint in $(-\infty,a)$ can be chosen to be $-\infty$ by the Hadamard-Caccioppoli theorem, because $\Phi$ is proper and $\det\sfJ_\Phi(s,t)\ne 0$ for all $s\in\dR$ and $t < 0$ due to the non-negativity of $\kp$.
This property of the mapping provides us with natural coordinates on the set $\Phi(\R\times (-\infty,a))$. Our aim is to extend these coordinates to the whole plane up to a certain set of Lebesgue measure zero. To this aim we define the cut-radius map $c\colon\dR\arr \dR_+$ by the property that the segment $\{\Phi(s,t')\colon t'\in(0,t)\}$ is the unique minimizer of the distance from $\Sg$ to $\Phi(s,t)$ if and only if $t\in [0,c(s))$. Adapting the argument of~\cite[Proposition 4.2.1]{SST03} we conclude that the cut-radius map $c$ is a continuous function.
The cut-locus
\[
{\rm Cut}\,(\Sg) := \{\Phi(s,c(s))\colon s\in\dR\}
\]
is a closed subset of $\Omg_+$ of Lebesgue measure zero. 
For any $s\in\dR$, the distance between $\Phi(s,c(s))$ and $\Sg$ is equal to $c(s)$ by construction of the cut-radius map. Thus, we conclude that
\begin{equation} \label{symmetry_cut_locus}
 \Phi(s_1,c(s_1))  =\Phi(s_2,c(s_2)) \in {\rm Cut}\,(\Sg) \quad \Rightarrow \quad c(s_1)=c(s_2).
\end{equation}
We recall that the restriction of the map $\Phi$ to the set
\begin{equation}\label{eq:U}
U := \{(s,t)\in\dR^2\colon -\infty < t < c(s)\}
\end{equation}
is a diffeomorphism onto $\Phi(U) = \dR^2\sm{\rm Cut}\,(\Sg)$ (see~\cite[Appendix 1]{S01}).
It also follows from~\cite[Lemma~4.2.3]{SST03} that
\begin{equation}\label{eq:boundc(s)}
	\kp(s)c(s) \le 1,\qquad \text{for all}\,\,s\in\dR.
\end{equation}
It will be convenient for the subsequent analysis to clarify
the values of $c(s)$ outside a compact subset of $\dR$
for the curve $\Sg$ satisfying Hypothesis~\ref{hyp}. This will be performed in the following proposition, whose proof is postponed to Appendix~\ref{app}.
\begin{prop}\label{assumption}
Let $\Sg\subset\dR^2$ be a curve as in Hypothesis~\ref{hyp} parametrized by the arc-length mapping $\gamma\colon\dR\arr\dR^2$ as above.
Then there exist $s_1,s_2\in\dR$, $s_1 < s_2$, such that
the following conditions hold:
\begin{myenum}
	\item $\supp\kp\subset [s_1,s_2]$;
	\item $h := -\gamma_2(s_1) = \gamma_2(s_2) > 0$; 
	\item The cut-radius map in $\dR\sm[s_1,s_2]$ is explicitly given by
	\begin{equation*}
	\begin{aligned}
		c(s) = \frac{h}{\sin\omg} + (s-s_2)\cot\omg,&\qquad s > s_2,\\
		c(s) = \frac{h}{\sin\omg} + (s_1-s)\cot\omg, &\qquad s < s_1.
	\end{aligned}
\end{equation*}
	In particular, outside a disk of a sufficiently large radius the cut-locus ${\rm Cut}\,(\Sg)$ coincides with the bisector ray of the broken line $\Sg_0$.
\end{myenum}
\end{prop}

\subsection{The Dirac operator with Lorentz scalar $\dl$-shell interaction}
\label{ssec:operator}

In the following, let $\Sigma$ be an unbounded $C^\infty$-curve as described at the beginning of Subsection~\ref{section_geometry};  it is not required that $\Sigma$ is a local perturbation of the broken line in~\eqref{eq:Sigma0} (but it should be a local perturbation of some broken line) and $\Sigma$ also does not have to satisfy Hypothesis~\ref{hyp}.
The aim of this subsection is to define the Dirac operator with a Lorentz-scalar $\dl$-shell interaction supported on $\Sg$ and to collect its basic properties. For this, we use a similar representation as in \cite{FHL24}, see also \cite{BEHT25}, where such operators have also been studied.

Let us recall the definition of the Hermitian $2\times 2$ Pauli matrices
\begin{equation}\label{eq:Pauli}
		\s_1 = \begin{pmatrix} 0 & 1 \\ 1 &0\end{pmatrix},\qquad
	\s_2 = \begin{pmatrix} 0 & -\ii\\ \ii &0\end{pmatrix},\qquad \s_3 = \begin{pmatrix} 1 & 0\\0&-1\end{pmatrix}.
\end{equation}
These matrices satisfy for all $i,j\in\{1,2,3\}$ the anti-commutation relation
$\s_i\s_j + \s_j\s_i = 2\dl_{ij} \sigma_0$, where $\dl_{ij}$ is the standard Kronecker symbol and $\sigma_0$ denotes the $2\times 2$-identity matrix.
In the following, we use for $x = (x_1, x_2)^\top \in \mathbb{C}^2$ the notation $\sigma \cdot x:=x_1\sigma_1+x_2\sigma_2$ and likewise $\sigma \cdot \nabla = \sigma_1 \partial_1 + \sigma_2 \partial_2$.
In order to define the operator, we introduce the mass $m > 0$
and the strength of the Lorentz-scalar interaction $\tau \in \mathbb{R} \setminus \{ -2, 0, 2 \}$. Consider the Dirac differential expression in $\dR^2$
\[
	\cD_m :=-\ii (\sigma\cdot\nabla)+m\s_3 \qquad\text{on}\,\,\dR^2.
\]
Next, we define for $\tau \neq \pm 2$ the matrix valued function $M_\tau\colon\Sigma\arr\dC^{2\times2}$ given by
\begin{equation} \label{def_M}
	M_\tau = \frac{1}{1-\frac{1}{4}\tau^2}\begin{pmatrix}
	1+\frac14\tau^2 & -\ii(\nu_1-\ii\nu_2)\tau\\
	\ii(\nu_1+\ii\nu_2)\tau & 1+\frac14\tau^2
	\end{pmatrix} = \frac{4 + \tau^2}{4 - \tau^2} \sigma_0 - \frac{4 \tau}{4 - \tau^2} \ii \sigma_3 (\sigma \cdot \nu).
\end{equation}
Note that $M_\tau$ has different signs in the anti-diagonal entries as, e.g., in \cite{FHL24} due to the fact that $\nu$ points inwards of $\Omega_+$. The anti-commutation relations for the Pauli matrices imply that $(\ii \sigma_3 (\sigma \cdot \nu))^2 = \sigma_0$ and thus, $M_\tau$ is pointwise invertible and the inverse is given by
\begin{equation} \label{M_inv}
  M_\tau^{-1}  = \frac{4 + \tau^2}{4 - \tau^2} \sigma_0 + \frac{4 \tau}{4 - \tau^2} \ii \sigma_3 (\sigma \cdot \nu).
\end{equation}
Now we are ready to introduce the Dirac operator in $L^2(\dR^2;\dC^2)$ with Lorentz scalar $\dl$-shell interaction supported on $\Sg$ of strength $\tau \in \mathbb{R} \setminus \{ -2, 0, 2\}$ via
\begin{equation}\label{eq:DiracOperator}
\begin{aligned}
	\sfD_{\tau,m} u &:= \cD_m u \quad \text{in } \mathbb{R}^2 \setminus \Sigma,\\
	\dom\sfD_{\tau,m} &:= \Big\{
	u = u_+\oplus u_-\in H^1(\Omg_+;\dC^2)\oplus H^1(\Omg_-;\dC^2)\colon u_-|_\Sg = M_\tau u_+|_\Sg\Big\}.
\end{aligned}	
\end{equation}
Note that the transmission conditions for $u \in \dom\sfD_{\tau,m}$ are equivalent to
\begin{equation} \label{transmission_conditions}
  \ii \sigma_3 (\sigma \cdot \nu ) (u_+|_\Sigma - u_-|_\Sigma) = \frac{\tau}{2} (u_+|_\Sigma + u_-|_\Sigma). 
\end{equation}
In the next proposition, we collect properties of the operator
$\sfD_{\tau,m}$ that are essential for our considerations.
\begin{prop}\label{essentials}
  Let $\tau \in \mathbb{R} \setminus \{ -2, 0, 2 \}$ and $m > 0$. Then the following statements hold:
\begin{myenum}
	\item The operator $\sfD_{\tau,m}$ is self-adjoint in $L^2(\dR^2;\dC^2)$.
	\item The essential spectrum of $\sfD_{\tau,m}$ is
	\begin{equation*}
	  \sess(\sfD_{\tau,m}) = \begin{cases} \left(-\infty,-m\left|\frac{\tau^2-4}{\tau^2+4}\right|\right]\cup
	\left[m\left|\frac{\tau^2-4}{\tau^2+4}\right|,+\infty\right), & \tau \in (-\infty, 0) \setminus \{-2\}, \\ (-\infty, -m] \cup [m, \infty),& \tau \in (0, \infty) \setminus \{2\}. \end{cases} 
    \end{equation*}
	\item For any $u\in \dom\sfD_{\tau,m}$ one has
	\begin{equation} \label{equation_form_square}
	  \begin{aligned}
		\frd_{\tau,m}[u]  :&= \|\sfD_{\tau,m}u\|^2_{L^2(\dR^2;\dC^2)}\\
		&= \int_{\dR^2\setminus\Sigma}\!\!|\nb u|^2\dd x + \frac{2m}{\tau}\int_\Sg|u_+ - u_-|^2\dd \mathcal{H}
		+\int_\Sg\frac{\kp}{2}\big(|u_+|^2-|u_-|^2\big)\dd \mathcal{H}
		+ m^2\int_{\dR^2}|u|^2\dd x.
	  \end{aligned}
	\end{equation}
	\item For $\tau \in (0, \infty) \setminus \{ 2 \}$ the discrete spectrum of $\sfD_{\tau,m}$ is empty.
	\item $\lambda \in \sigma(\sfD_{\tau,m})$ if and only if $-\lambda \in \sigma(\sfD_{\tau,m})$ and, moreover, for any $\lambda \in \mathbb{R}$ one has
	\[
	\dim\ker(\sfD_{\tau,m}-\lm) = \dim\ker(\sfD_{\tau,m}+\lm).
	\]
\end{myenum}	
\end{prop}
\begin{proof}
	Items~(i) and~(ii) follow from \cite[Theorem~2.3 and Corollary~2.5]{BEHT25} when substituting $c=1$ in \cite{BEHT25}. To show~(iii), we follow  a similar strategy as in the proof of \cite[Theorem~3.2]{FHL24}, see also \cite[Proposition~3.1]{HOP18}, \cite[Theorem~1.5~(iv)]{ALTR17}, or \cite[Proposition~14]{LO18} for similar formulas and constructions. First, as the parametrization $\gamma$ is in the clockwise orientation for $\Omega_+$ and the same convention is used in \cite{FHL24}, one gets in the same way as in \cite[Lemma~3.1]{FHL24} that for all $f, g \in H^1(\Omega_\pm; \mathbb{C}^2)$
	\begin{equation} \label{integration_by_parts}
	  \begin{split}
	    \int_{\Omega_\pm} (\sigma \cdot \nabla) f \overline{(\sigma \cdot \nabla) g} \dd x = \int_{\Omega_\pm} \nabla f \cdot \overline{\nabla g} \dd x \mp \langle f|_\Sigma, \ii \sigma_3 \partial_t (g|_\Sigma) \rangle_{H^{1/2}(\Sigma; \mathbb{C}^2) \times H^{-1/2}(\Sigma; \mathbb{C}^2)}
	  \end{split}
	\end{equation}
	holds, where $\langle \cdot, \cdot \rangle_{H^{1/2}(\Sigma; \mathbb{C}^2) \times H^{-1/2}(\Sigma; \mathbb{C}^2)}$ is the sesquilinear duality product in $H^{1/2}(\Sigma; \mathbb{C}^2) \times H^{-1/2}(\Sigma; \mathbb{C}^2)$ and $(\partial_t g|_\Sigma) (\gamma(s)) = (g|_\Sigma \circ \gamma)'(s)$ is the tangential derivative of $g|_\Sigma$ along $\Sigma$; for the sign of the boundary term, one has to take into account that for $\Omega_-$ the orientation of $\gamma$ is counter clockwise. We also remark that 
	we implicitly mean that
	$\nb f\cdot\ov{\nb g} = \nb f_1\cdot\ov{\nb g_1} + \nb f_2\cdot\ov{\nb g_2}$ for $f= (f_1,f_2)^\top\in H^1(\Omg_\pm;\dC^2)$ and $g= (g_1,g_2)^\top\in H^1(\Omg_\pm;\dC^2)$.
	
	In the following, let $u\in \dom\sfD_{\tau,m}$ be fixed. Then,  \eqref{integration_by_parts} yields
\begin{equation}\label{e-startnorm}
\begin{aligned}
    \|\sfD_{\tau,m}u&\|^2_{L^2(\dR^2 ;\dC^2)}=\|(-\ii\sigma\cdot\nabla )u+m\sigma_3 u\|^2_{L^2(\dR^2 \setminus \Sigma;\dC^2)} \\
    &=\|(\sigma\cdot\nabla )u\|^2_{L^2(\dR^2 \setminus \Sigma;\dC^2))}+m^2\|u\|^2_{L^2(\dR^2;\dC^2)}+2\Re\left(\langle-\ii (\sigma\cdot\nabla )u,m\sigma_3 u\rangle_{L^2(\dR^2 \setminus \Sigma;\dC^2)}\right) \\
    &=\|\nabla  u\|^2_{L^2(\dR^2 \setminus \Sigma;\dC^2\otimes\dC^2)}+m^2\|u\|^2_{L^2(\RR^2;\dC^2)}+2\Re\left(\langle-\ii (\sigma\cdot\nabla )u,m\sigma_3 u\rangle_{L^2(\dR^2 \setminus \Sigma;\dC^2)}\right) \\
    &\qquad+\langle u_-|_\Sigma,\ii \sigma_3 \partial_t (u_-|_\Sigma)\rangle_{H^{1/2}(\Sigma; \mathbb{C}^2)\times H^{-1/2}(\Sigma; \mathbb{C}^2)}\\
    &\qquad -\langle u_+|_\Sigma,\ii \sigma_3 \partial_t (u_+|_\Sigma)\rangle_{H^{1/2}(\Sigma; \mathbb{C}^2)\times H^{-1/2}(\Sigma; \mathbb{C}^2)}.
\end{aligned}
\end{equation}
Recall that $M_\tau$ is defined by~\eqref{def_M} and note that $\sigma_3(\sigma\cdot\nu)=-(\sigma\cdot\nu)\sigma_3$ and~\eqref{M_inv} imply
\begin{equation*}
   M_\tau \sigma_3 = \sigma_3 \left( \frac{4+\tau^2}{4-\tau^2}\sigma_0+\frac{4\tau}{4-\tau^2}\ii\sigma_3(\sigma\cdot\nu) \right) = \sigma_3 M_\tau^{-1}.
\end{equation*}
Moreover, Frenet's formula \eqref{Frenet} and the convention $\nu = (\gamma_2', -\gamma_1')^\top$ yield
\begin{equation*}
  \partial_t (\sigma \cdot \nu) = \sigma_1 \gamma_2'' - \sigma_2 \gamma_1'' = \kappa ( \sigma_1 \nu_2 - \sigma_2 \nu_1) = \ii \kappa \sigma_3 (\sigma \cdot \nu).
\end{equation*}
By combining the last two displayed equations with the product rule and the fact that $M_\tau$ is self-adjoint, one finds that
\begin{equation}\label{e-boundary0s}
\begin{aligned}
    \big\langle & u_-|_\Sigma,\ii \sigma_3\partial_t (u_-|_\Sigma)\big\rangle_{H^{1/2}(\Sigma; \mathbb{C}^2)\times H^{-1/2}(\Sigma; \mathbb{C}^2)}
    - \big\langle u_+|_\Sigma,\ii \sigma_3\partial_t (u_+|_\Sigma)\big\rangle_{H^{1/2}(\Sigma; \mathbb{C}^2)\times H^{-1/2}(\Sigma; \mathbb{C}^2)}  \\
    &= \big\langle M_\tau u_+|_\Sigma,\ii \sigma_3\partial_t (u_-|_\Sigma)\big\rangle_{H^{1/2}(\Sigma; \mathbb{C}^2)\times H^{-1/2}(\Sigma; \mathbb{C}^2)} \\
    &\hspace{6em} -\big\langle u_+|_\Sigma,\ii \sigma_3\partial_t ( M_\tau^{-1} u_-|_\Sigma)\big\rangle_{H^{1/2}(\Sigma; \mathbb{C}^2)\times H^{-1/2}(\Sigma; \mathbb{C}^2)} \\
     &= \big\langle M_\tau u_+|_\Sigma,\ii \sigma_3\partial_t (u_-|_\Sigma)\big\rangle_{H^{1/2}(\Sigma; \mathbb{C}^2)\times H^{-1/2}(\Sigma; \mathbb{C}^2)} \\
    &\hspace{1em} -\left\langle u_+|_\Sigma,\ii \sigma_3 M_\tau^{-1} \partial_t (u_-|_\Sigma) + \ii \sigma_3 \tfrac{4 \tau}{4-\tau^2}\ii \sigma_3 \partial_t (\sigma \cdot \nu) u_-|_\Sigma\right\rangle_{H^{1/2}(\Sigma; \mathbb{C}^2)\times H^{-1/2}(\Sigma; \mathbb{C}^2)} \\
    &= \big\langle M_\tau u_+|_\Sigma,\ii \sigma_3\partial_t (u_-|_\Sigma)\big\rangle_{H^{1/2}(\Sigma; \mathbb{C}^2)\times H^{-1/2}(\Sigma; \mathbb{C}^2)} \\
    &\hspace{6em} -\big\langle M_\tau u_+|_\Sigma,\ii\sigma_3\partial_t (u_-|_\Sigma)\big\rangle_{H^{1/2}(\Sigma; \mathbb{C}^2)\times H^{-1/2}(\Sigma; \mathbb{C}^2)} \\
     &\hspace{6em} +\left\langle u_+|_\Sigma,\ii\sigma_3 \tfrac{4 \tau}{4-\tau^2} \kappa (\sigma \cdot \nu) u_-|_\Sigma \right\rangle_{H^{1/2}(\Sigma; \mathbb{C}^2)\times H^{-1/2}(\Sigma; \mathbb{C}^2)} \\
    &= \frac{1}{2} \Big(-\big\langle \kappa M_\tau u_+|_\Sigma,u_-|_\Sigma\big\rangle_{H^{1/2}(\Sigma; \mathbb{C}^2)\times H^{-1/2}(\Sigma; \mathbb{C}^2)}  +\big\langle \kappa  u_+|_\Sigma, M_\tau^{-1} u_-|_\Sigma\big\rangle_{H^{1/2}(\Sigma; \mathbb{C}^2)\times H^{-1/2}(\Sigma; \mathbb{C}^2)} \Big) \\
    &= \int_\Sg\frac{\kp}{2}\big(|u_+|^2-|u_-|^2\big)\dd \mathcal{H}.
\end{aligned}
\end{equation}
Substituting \eqref{e-boundary0s} into \eqref{e-startnorm} it remains to compute
$
2\Re(\langle-\ii (\sigma\cdot\nabla )u,m\sigma_3 u\rangle_{L^2(\dR^2 \setminus \Sigma;\dC^2)})
$
to deduce the claimed expression in~\eqref{equation_form_square}.

Next, note that due to the fact that $\nu$ is pointing inwards of $\Omega_+$ and outwards of $\Omega_-$ one has
\begin{align*}
    \big\langle-\ii (\sigma\cdot\nabla )u_{\pm},m\sigma_3 u_{\pm}\big\rangle_{L^2(\Omega_{\pm}; \mathbb{C}^2)}&+\big\langle m\sigma_3 u_{\pm},-\ii (\sigma\cdot\nabla )u_{\pm}\big\rangle_{L^2(\Omega_{\pm}; \mathbb{C}^2)}\\
    &=\mp\big\langle-\ii (\sigma\cdot\nu)u_{\pm}|_\Sigma,m\sigma_3 u_{\pm}|_\Sigma\big\rangle_{L^2(\Sigma; \mathbb{C}^2)}.
\end{align*}
This identity
and the fact that $M_\tau$ is self-adjoint and commutes with $-\ii\sigma_3(\sigma\cdot\nu)$ allow for the following simplification: 
\begin{equation} \label{e-boundary1s}
\begin{aligned}
2\Re&\left(\langle-\ii (\sigma\cdot\nabla )u,m\sigma_3 u\rangle_{L^2(\dR^2 \setminus \Sigma;\dC^2)}\right) \\
&=\big\langle-\ii (\sigma\cdot\nabla )u,m\sigma_3 u\big\rangle_{L^2(\dR^2 \setminus \Sigma;\dC^2)}
 +\big\langle m\sigma_3 u,-\ii (\sigma\cdot\nabla )u\big\rangle_{L^2(\dR^2 \setminus \Sigma;\dC^2)} \\
 &=m\Bigl[-\big\langle-\ii \sigma_3(\sigma\cdot\nu)u_+|_\Sigma,u_+|_\Sigma\big\rangle_{L^2(\Sigma; \mathbb{C}^2)}
 +\big\langle-\ii \sigma_3(\sigma\cdot\nu)u_-|_\Sigma,u_-|_\Sigma\big\rangle_{L^2(\Sigma; \mathbb{C}^2)} \\
 &\hspace{4em}+\big\langle u_+|_\Sigma,-\ii \sigma_3(\sigma\cdot\nu)M_\tau u_+|_\Sigma\big\rangle_{L^2(\Sigma; \mathbb{C}^2)}\\
  &\hspace{4em}-\big\langle u_+|_\Sigma,-\ii \sigma_3(\sigma\cdot\nu) M_\tau u_+|_\Sigma\big\rangle_{L^2(\Sigma; \mathbb{C}^2)}\Bigr] \\
 &=m\Bigl[-\big\langle-\ii \sigma_3(\sigma\cdot\nu)u_+|_\Sigma,u_+|_\Sigma\big\rangle_{L^2(\Sigma; \mathbb{C}^2)}
 +\big\langle-\ii \sigma_3(\sigma\cdot\nu )u_-|_\Sigma,u_-|_\Sigma\big\rangle_{L^2(\Sigma; \mathbb{C}^2)} \\
 &\hspace{4em}+\big\langle u_-|_\Sigma,-\ii \sigma_3(\sigma\cdot\nu)u_+|_\Sigma\big\rangle_{L^2(\Sigma; \mathbb{C}^2)}
 -\big\langle u_+|_\Sigma,-\ii \sigma_3(\sigma\cdot\nu)u_-|_\Sigma\big\rangle_{L^2(\Sigma; \mathbb{C}^2)}\Bigr] \\
 &=m \, \Big\langle \ii \sigma_3(\sigma\cdot\nu)(u_+|_\Sigma-u_-|_\Sigma),(u_+|_\Sigma+u_-|_\Sigma)\Big\rangle_{L^2(\Sigma; \mathbb{C}^2)} \\
 &=m\left\langle \ii \sigma_3(\sigma\cdot\nu)(u_+|_\Sigma-u_-|_\Sigma),\frac{2}{\tau}\ii \sigma_3(\sigma\cdot\nu)(u_+|_\Sigma-u_-|_\Sigma)\right\rangle_{L^2(\Sigma; \mathbb{C}^2)}\\
 &=\frac{2m}{\tau}\big\|u_+|_\Sigma-u_-|_\Sigma\big\|^2_{L^2(\Sigma; \mathbb{C}^2)},
\end{aligned}
\end{equation}
where the second to last equality is due to the transmission condition in equation \eqref{transmission_conditions}. By using~\eqref{e-boundary0s} and~\eqref{e-boundary1s} in~\eqref{e-startnorm} one finds that the claimed formula for $\|\sfD_{\tau,m}u\|^2_{L^2(\dR^2;\dC^2)}$ is true.

(iv)  Let  $u\in \dom\sfD_{\tau,m}$. By using \eqref{e-boundary1s} in~\eqref{e-startnorm} one finds for $\tau \in (0, \infty) \setminus \{ 2 \}$ that
\begin{equation*}
\begin{aligned}
    \|\sfD_{\tau,m}u&\|^2_{L^2(\dR^2;\dC^2)}=\|(-\ii\sigma\cdot\nabla )u+m\sigma_3 u\|^2_{L^2(\dR^2 \setminus \Sigma;\dC^2)} \\
    &=\|(\sigma\cdot\nabla )u\|^2_{L^2(\dR^2 \setminus \Sigma;\dC^2)}+m^2\|u\|^2_{L^2(\dR^2;\dC^2)}+2\Re\left(\langle-\ii (\sigma\cdot\nabla )u,m\sigma_3 u\rangle_{L^2(\dR^2 \setminus \Sigma;\dC^2)}\right) \\
    &=\|(\sigma\cdot\nabla )u\|^2_{L^2(\dR^2 \setminus \Sigma;\dC^2)}+m^2\|u\|^2_{L^2(\dR^2;\dC^2)} + \frac{2m}{\tau}\big\|u_+|_\Sigma-u_-|_\Sigma\big\|^2_{L^2(\Sigma; \mathbb{C}^2)} \\
    &\geq m^2 \| u \|^2_{L^2(\dR^2; \mathbb{C}^2)}.
\end{aligned}
\end{equation*}
This shows that $\sfD_{\tau,m}^2$ does not have eigenvalues that are smaller than $m^2$. In particular, this implies $\sigma(\sfD_{\tau,m}) \cap (-m,m) = \emptyset$.
Since $\sess(\sfD_{\tau,m}) = (-\infty, -m] \cup [m, \infty)$ by (ii), it follows that $\sfD_{\tau,m}$ does not have discrete eigenvalues.

(v) In order to show the claim, we follow closely arguments from \cite[Proposition~4.8~(ii)]{BHOBP}, see also \cite[Theorem~2.3~(b)]{HOP18}, \cite[Theorem~1.5~(ii)]{ALTR17}, or \cite[Proposition~6~(ii)]{LO18} for similar proofs. Consider the antilinear charge conjugation operator $C u  =\sigma_1 \overline{u}$. 
Then one has $C^2 u = u$ for all $u \in L^2(\mathbb{R}^2; \mathbb{C}^2)$. We show that 
  \begin{equation} \label{charge_conjugation}
    C \sfD_{\tau,m} = -\sfD_{\tau,m} C,
  \end{equation}
  which yields then the claim of statement~(v). To verify~\eqref{charge_conjugation}, note first by taking the complex 
  conjugate of the transmission condition in~\eqref{transmission_conditions} that $u \in \dom \sfD_{\tau,m}$ if and only if
  \begin{equation}\label{complex_conjugate}
    -\ii \sigma_3 (\overline{\sigma}\cdot \nu)\big( \overline{u_+|_\Sigma} - \overline{u_-|_\Sigma} \big) 
    = \frac{\tau}{2}\big(\overline{u_+|_\Sigma} + \overline{u_-|_\Sigma}\big),
  \end{equation}
  where $\overline\sigma =(\overline{\sigma_1},\overline{\sigma_2})$ and $\overline{\sigma_j}$ is the matrix  with the complex conjugate entries of $\sigma_j$.
  By multiplying this equation with $\sigma_1$ from the left and using the anti-commutation relations for the Pauli matrices, $\overline{\sigma_1} = \sigma_1$, and $\overline{\sigma_2} = -\sigma_2$ 
  one finds that \eqref{complex_conjugate} is equivalent to
  \begin{equation*}
    \ii \sigma_3 (\sigma \cdot \nu)\big( \sigma_1 \overline{u_+|_\Sigma} - \sigma_1 \overline{u_-|_\Sigma}\big) = \frac{\tau}{2}\big(\sigma_1 \overline{u_+|_\Sigma} + \sigma_1 \overline{u_-|_\Sigma}\big),
  \end{equation*}
  \ie $Cu \in \dom \sfD_{\tau,m}$. This shows that $C( \dom  \sfD_{\tau,m}) = \dom  \sfD_{\tau,m}$. Moreover, using again the anti-commutation relations for the Pauli matrices and $\overline{\sigma_2} = -\sigma_2$
  one gets
  \begin{equation*}
    \begin{split}
      (-\ii \sigma \cdot \nabla + m \sigma_3) C u &= (-\ii \sigma \cdot \nabla + m \sigma_3) \sigma_1 \overline{u}\\
      &= \sigma_1 (-\ii \overline{\sigma} \cdot \nabla - m \sigma_3) \overline{u} \\
      &= -\sigma_1 \overline{(-\ii \sigma \cdot \nabla + m \sigma_3) u} \\
      &= - C \bigl( -\ii \sigma \cdot \nabla + m \sigma_3) u \bigr),
    \end{split}
  \end{equation*}
  which implies \eqref{charge_conjugation} and finishes the proof.
\end{proof}

\begin{remark} \label{remark_confinement}
  First, note that the claims in items~(i) and (ii) in Proposition~\ref{essentials} are also true for $\tau = \pm 2$, as there is no restriction in \cite{BEHT25} on $\tau$. Moreover, the arguments leading to Proposition~\ref{essentials}~(iv) can also be done for $\tau = 2$.
  
  In particular, this implies that in the confinement case $\tau = \pm 2$ the operator $\sfD_{\tau,m}$ does not have geometrically induced discrete eigenvalues. Indeed, for $\tau = 2$ one has, as for any $\tau \geq 0$, that $\sigma(\sfD_{2,m}) = \sess(\sfD_{2,m}) = (-\infty,-m]\cup[m,\infty)$, while for $\tau = -2$ one has by Proposition~\ref{essentials}~(ii) that $\sess(\sfD_{-2,m}) = \mathbb{R}$. Hence, in both cases $\tau = \pm 2$ the discrete spectrum of $\sfD_{\tau,m}$ is empty.
\end{remark}

\section{Finiteness of the discrete spectrum}
\label{sec:finiteness}

Our main goal in this section is to prove that the operator $\sfD_{\tau,m}$ defined by~\eqref{eq:DiracOperator} has finitely many discrete eigenvalues. 
We remark that in this section we deal with a general smooth curve being a local deformation of the broken line. We do not assume that the curve satisfies Hypothesis~\ref{hyp}. Recall that for a closed, densely defined, semibounded symmetric sesquilinear form $\frt$ with domain $\dom\frt\subset\mathcal{H}$ in an infinite-dimensional Hilbert space $\mathcal{H}$, the quantity $N(\frt,\lambda)$, $\lambda\in\RR$, is defined by~\eqref{eigenvalue_counting_form}. The main result in this section is the following:

\begin{theorem}\label{finite prop}
Let $\tau < 0$, $m > 0$, and let $\frd_{\tau,m}$ be as in~\eqref{equation_form_square}.
Then there holds
\begin{equation} \label{form_finite}
N\left(\frd_{\tau,m},m^2\left(\tfrac{\tau^2-4}{\tau^2+4}\right)^2\right)<\infty.
\end{equation}
In particular, the operator $\sfD_{\tau,m}$ has finite discrete spectrum.
\end{theorem}

In the proof of Theorem~\ref{finite prop} the following auxiliary result about a one-dimensional operator will be used:

\begin{lemma} \label{lemma_1D}
  For $m>0$ and $\mu \in \mathbb{R} \setminus \{0, \pm 2 \}$ the quadratic form $\frp_\mu$ in $L^2(\mathbb{R}; \mathbb{C})$ defined by
  \[
  \begin{aligned}
\frp_\mu[u]&:=\int_\RR|u'|^2\dd t-\frac{2m}{|\mu|}|u(0^+)-u(0^-)|^2,
\\
\dom\frp_\mu&:=\left\{u\in H^1(\RR\setminus\{0\}; \mathbb{C})\colon u(0^-)=\frac{2 - \mu}{2 + \mu} u(0^+) \right\},
\end{aligned}
\]
is densely defined, semibounded from below, and closed. The associated self-adjoint operator $\sfP_\mu$ satisfies $\sigma_\textup{ess}(\sfP_\mu) = [0, \infty)$ and $\sfP_\mu$ has exactly one simple negative eigenvalue
\begin{equation} \label{equation_eigenfunction}
E_1(\sfP_\mu) = -\frac{16 m^2 \mu^2 }{(\mu^2+4)^2}
\end{equation}
with normalized eigenfunction
\begin{equation} \label{def_Phi_tau}
\Phi_{\mu}(t)=\begin{cases} c_{\mu}\exp\left(-\sqrt{-E_1(\sfP_\mu)}t\right), &\text{ if } t>0, \\ c_{\mu}\frac{2-\mu}{2+\mu} \exp\left(\sqrt{-E_1(\sfP_\mu)}t\right), &\text{ if }t<0,\end{cases}
\end{equation}
where $c_\mu\in\RR$ is a suitable normalization constant.
\end{lemma}
\begin{proof}
  First, as $\dom\frp_\mu$ is  a closed subspace of $H^1(\mathbb{R} \setminus \{ 0 \})$ and the norm associated with $\frp_\mu$ is equivalent to the $H^1$-norm, $\frp_\mu$ is closed. A direct calculation shows that the self-adjoint operator $\sfP_\mu$ that is associated with $\frp_\mu$ is given by
  \begin{equation*}
    \begin{split}
      \sfP_\mu u &= -u'' \quad \text{in } \mathbb{R} \setminus \{ 0 \}, \\
      \dom \sfP_\mu &= \bigg\{ u \in H^2(\mathbb{R}^2 \setminus \{ 0 \})\colon u(0^-)=\frac{2 - \mu}{2 + \mu} u(0^+), \\
      &\qquad \qquad  \qquad \qquad \qquad \frac{2 - \mu}{2 + \mu} u'(0^-) - u'(0^+) = \frac{2m}{|\mu|} (u(0^+) - u(0^-)) \left( 1 - \frac{2 - \mu}{2 + \mu} \right) \bigg\}.
    \end{split}
  \end{equation*}
  In particular, $\sfP_\mu$ is a self-adjoint extension of the symmetric operator
  \begin{equation*}
    \sfS u = -u'', \qquad \dom \sfS = H^2_0(\mathbb{R} \setminus \{ 0 \}).
  \end{equation*}
Therefore, by \cite[Theorem~3.1]{CH93} one has $\sigma_\textup{ess}(\sfP_\mu) = [0, \infty)$. 
 Finally, a direct calculation shows that any eigenfunction $\Phi$ associated with an eigenvalue $E < 0$ of $\sfP_\mu$ must be of the form
 \begin{equation*}
   \Phi(t) = \begin{cases} c_1 \exp\left(-\sqrt{-E}t\right), &\text{ if } t>0, \\ c_2 \exp\left(\sqrt{-E}t\right), &\text{ if }t<0,\end{cases}
 \end{equation*}
 where $c_1, c_2$ are suitable constants. The first transmission condition for $\Phi \in \dom \sfP_\mu$ yields $c_2 = \frac{2 - \mu}{2 + \mu} c_1$, while the second transmission condition holds, if
 \begin{equation*}
   \sqrt{-E} \left[ \left( \frac{2-\mu}{2+\mu} \right)^2 + 1 \right] c_1 = \frac{2m}{|\mu|} \left( 1 - \frac{2-\mu}{2+\mu} \right)^2 c_1.
 \end{equation*}
 This equation has exactly one solution given by $E = E_1(\sfP_\mu) = -\frac{16 m^2 \mu^2 }{(\mu^2+4)^2}$, which shows that~\eqref{equation_eigenfunction} and~\eqref{def_Phi_tau} are true.
\end{proof}

In the proof of Theorem~\ref{finite prop} we will make use of projections associated with the eigenfunctions in~\eqref{def_Phi_tau}. We denote by
\begin{equation*}
  \Gamma_0 = \{ (0, x_2): x_2 \in \mathbb{R} \}, \qquad \check{\Omega}_\pm := \{ (x_1,x_2) \in \mathbb{R}^2\colon \pm x_1 \geq 0 \},
\end{equation*}
and we write for $u \in H^1(\mathbb{R}^2 \setminus \Gamma_0; \mathbb{C}^d)$, $d \in \mathbb{N}$,
\begin{equation*}
  u_\pm := u |_{\check{\Omega}_\pm}.
\end{equation*}
The proof of the following result is straightforward taking the properties of $\Phi_\mu$ into account and is left to the reader.

\begin{lemma} \label{lemma_projections} 
  Let $\mu \in \mathbb{R} \setminus \{0, \pm 2 \}$, $\Phi_\mu$ be given by~\eqref{def_Phi_tau} and, for $u\in L^2(\RR^2;\CC)$, denote 
  \begin{equation} \label{def_psi_tau}
    \psi_{\mu,u} := \int_\mathbb{R} u(t,\cdot) \overline{\Phi_\mu(t)} \dd t.
  \end{equation}
  Then, the mappings
\begin{equation} \label{def_projections1}
  \Pi_\mu: L^2(\mathbb{R}^2; \mathbb{C}) \rightarrow L^2(\mathbb{R}^2; \mathbb{C}), \quad (\Pi_\mu u)(x_1,x_2) = \Phi_\mu(x_1) \psi_{\mu,u}(x_2)
\end{equation}
and
\begin{equation} \label{def_projections2}
  \mathbb{P}_\mu: L^2(\mathbb{R}^2; \mathbb{C}) \rightarrow L^2(\mathbb{R}^2; \mathbb{C}), \quad \mathbb{P}_\mu = I - \Pi_\mu,
\end{equation}
are well defined and bounded and they satisfy the following properties:
\begin{itemize}
  \item[(i)] For $u \in H^1(\mathbb{R}^2 \setminus \Gamma_0; \mathbb{C})$ one has $\Pi_\mu u, \mathbb{P}_\mu u \in  H^1(\mathbb{R}^2 \setminus \Gamma_0; \mathbb{C})$ and $(\Pi_\mu u)_-|_{\Gamma_0} = \frac{2 - \mu}{2 + \mu} (\Pi_\mu u)_+|_{\Gamma_0}$.
  \item[(ii)] For almost all $x_2 \in \mathbb{R}$ one has $\Pi_\mu u(\cdot,x_2)\perp_{L^2(\RR)}\mathbb{P}_\mu u(\cdot,x_2)$.
  \item[(iii)] For $u \in H^1(\mathbb{R}^2 \setminus \Gamma_0; \mathbb{C})$ satisfying $(\mathbb{P}_\mu u)_-|_{\Gamma_0} = \frac{2 - \mu}{2 + \mu} (\mathbb{P}_\mu u)_+|_{\Gamma_0}$ one has for almost all $x_2 \in \mathbb{R}$
  \begin{equation*}
    \int_\RR|\partial_1 \mathbb{P}_\mu u(x_1, x_2)|^2\dd x_1 - \frac{2m}{|\mu|}|(\mathbb{P}_\mu u)_+(0, x_2)-(\mathbb{P}_\mu u)_-(0,x_2)|^2 \geq 0.
  \end{equation*}
\end{itemize}
\end{lemma}

Now we are prepared to show Theorem~\ref{finite prop}.

\begin{proof}[Proof of Theorem~\ref{finite prop}]
First, recall that for a closed, densely defined, semibounded symmetric sesquilinear form $\frt$ with domain $\dom\frt\subset\mathcal{H}$ in an infinite-dimensional Hilbert space $\mathcal{H}$ the quantity $\lambda_n(\frt)$, $n\in\NN$, is defined by~\eqref{min_max_form}.
Moreover, throughout the proof we will frequently use the notations 
\[
\RR^2_+=\RR\times(0,\infty),\quad \RR^2_-=\RR\times(-\infty,0).
\]
For the sake of convenience, set
\[
\fra[u]:=\frd_{\tau,m}[u]-m^2\|u\|_{L^2(\R^2)}^2,
\qquad \dom\fra:=\dom\frd_{\tau,m},
\]
and
\[
\lambda_\tau:=m^2\left(\frac{\tau^2-4}{\tau^2+4}\right)^2-m^2 = -\frac{16 m^2 \tau^2 }{(\tau^2+4)^2}.	
\]
The statement of the theorem is thus equivalent to $N(\fra,\lm_\tau) < \infty$. Throughout the analysis we assume for definiteness that $\Sg$ is a local deformation of the broken line $\Sg_0$ defined in~\eqref{eq:Sigma0}.

The rest of the proof is divided into steps for convenience of the reader. In \textit{Step~1} the IMS formula is used to ``partition'' and estimate the quadratic form $\fra$ from below in the sense of ordering of forms into parts, where the interaction support is part of the straight line, and into parts, where the interaction is compactly supported. In \textit{Step~2} the finiteness of the discrete spectrum associated with the part, where the interaction is compactly supported, is shown. In \textit{Step~3} the finiteness of the discrete spectrum associated with the part, where the interaction is part of a straight line, is verified. Finally, in \textit{Step~4} these results are combined to prove the claim of the theorem.

\noindent\emph{Step 1: }
In the following we will ``partition'' and estimate the quadratic form $\fra$ from below in the sense of ordering of forms. First, we impose a Neumann boundary condition on the $x_1$-axis. It will be convenient to introduce the notation $\Sg_\pm:=\Sg\cap\R^2_\pm$. More precisely we define the quadratic form $\fra_{\rm N}:=\fra_{\rm N}^+\oplus\fra_{\rm N}^-$ in $L^2(\dR^2;\dC^2)$, where
\begin{equation*}
\begin{aligned}
\fra_{\rm N}^\pm[u]
&:= \int_{\dR^2_\pm}|\nb u|^2\dd x + \frac{2m}{\tau}\int_{\Sg_\pm}|u_+ - u_-|^2\dd \mathcal{H}
		+\int_{\Sg_\pm}\frac{\kp}{2}\big(|u_+|^2-|u_-|^2\big)\dd \mathcal{H},\\
\dom\fra_{\rm N}^\pm &:=\Big\{
	u \in H^1(\Omg_+\cap\R^2_\pm; \mathbb{C}^2)\oplus H^1(\Omg_-\cap\R^2_\pm; \mathbb{C}^2) \colon u_-|_{\Sg_\pm} = M_\tau u_+|_{\Sg_\pm}\Big\}.
\end{aligned}
\end{equation*}
Due to $\fra_{\rm N}[u]=\fra[u]$ for $u\in\dom\fra$ and $\dom\fra\subset\dom\fra_{\rm N}$, we have $\lambda_n(\fra_{\rm N})\le\lambda_n(\fra)$ for all $n\in\NN$ and therefore for any $\lambda\in\RR$
\begin{equation}\label{N-2}
N(\fra,\lambda)\le N(\fra_{\rm N},\lambda).
\end{equation}
Moreover, we have 
\begin{equation}\label{N-1}
N(\fra_{\rm N},\lambda)=N(\fra_{\rm N}^-,\lambda)+N(\fra_{\rm N}^+,\lambda).
\end{equation}
It is sufficient to prove that $N(\fra_{\rm N}^\pm,\lambda_\tau)$ are finite. Since the proofs in both half-planes are analogous, below we only prove the finiteness of $N(\fra_{\rm N}^+,\lambda_\tau)$. We would like to ''partition''  $\fra_{\rm N}^+$ into forms which either have $\delta$-shell interactions on the arc or on the ray or have no interaction at all. For that purpose we choose a so-called IMS partition of unity \cite[Section~3.1]{cycon}. Namely, let $C^\infty$-smooth functions $\chi_0,\chi_1\colon(0,\infty)\to[0,1]$ be such that $\chi_0(t)=1$ for $0<t<1$ and $\chi_0(t)=0$ for $t>2$ and $\chi_0^2+\chi_1^2=1$. We define $\chi_{j,T}(x_1,x_2):=\chi_j(x_2/T)$ with $T>0$ and $j\in\{0,1\}$. Clearly, $\partial_1 \chi_{j,T}=0$. Moreover, $\chi_0(t)^2 + \chi_1(t)^2=1$ implies $2 (\chi_{0,T} \partial_2 \chi_{0,T} + \chi_{1,T} \partial_2 \chi_{1,T}) = 0$.
Hence, by straightforward calculations we get for any $u\in\dom\fra_{\rm N}^+$
\[
\fra_{\rm N}^+[u]=\fra_{\rm N}^+\big[\chi_{0,T}u\big]+\fra_{\rm N}^+[\chi_{1,T}u]-\int_{\R^2_+}\big(|\partial_2 \chi_{0,T}(x)|^2+|\partial_2 \chi_{1,T}(x)|^2\big)|u|^2\dd x.
\]
Let us introduce the potential
\begin{equation}
V_T(x_1,x_2):=|\partial_2 \chi_{0,T}(x)|^2+|\partial_2 \chi_{1,T}(x)|^2=\frac{1}{T^2}\left(\left|\chi_{0}'\left(\frac{x_2}{T}\right)\right|^2+\left|\chi_{1}'\left(\frac{x_2}{T}\right)\right|^2\right)\label{potential}
\end{equation}
and the subsets 
\[
\cA_T:=\{(x_1,x_2)\in\R^2_+\colon 0<x_2<2T\},\qquad
\cB_T:=\{(x_1,x_2)\in\R^2_+\colon x_2>T\}.
\]
	\begin{figure}[h]
\includegraphics[width=7cm, keepaspectratio]{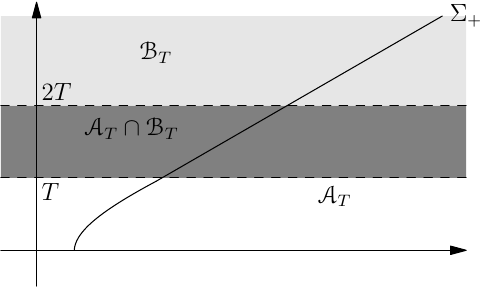}
	\caption{Schematic picture of the sets $\cA_T$ and $\cB_T$. The region $\cA_T \setminus \cB_{T}$ is depicted in white, $\cB_T \setminus \cA_T$ in light gray, and $\cA_T \cap \cB_T$  in gray.}
	\end{figure}
In the following, we choose $T>0$ large enough so that $\Sg_+\cap\cB_T$ is part of a straight line.
Then $\kappa=0$ on $\Sg_+ \cap \cB_T$ and thus,
\begin{align*}
\begin{split}
\fra_{\rm N}^+[u]=&\int_{\cA_T}\big(|\nabla(  \chi_{0,T}u)|^2-V_T|\chi_{0,T}u|^2\big)\dd x+\frac{2m}{\tau}\int_{\Sigma_+\cap \cA_T}|\chi_{0,T}(u_+-u_-)|^2\dd \mathcal{H} 
\\
&+\int_{\Sg_+\cap\cA_T}\frac{\kp}{2}\big(|\chi_{0,T}u_+|^2-|\chi_{0,T}u_-|^2\big)\dd \mathcal{H}
\\
&+\int_{\cB_T}\big(|\nabla(\chi_{1,T} u )|^2-V_T|\chi_{1,T} u|^2\big)\dd x+\frac{2m}{\tau}\int_{\Sigma_+\cap \cB_T}|\chi_{1,T}(u_+-u_-)|^2\dd \mathcal{H} .
\end{split}
\end{align*}
Now let us introduce the forms
\begin{align*}
\begin{split}
\frq_{\cA_T}[u]&:=\int_{\cA_T}\big(|\nabla u |^2-V_T|u|^2\big)\dd x+\frac{2m}{\tau}\int_{\Sigma_+\cap \cA_T}|u_+-u_-|^2\dd \mathcal{H} +\int_{\Sg_+\cap\cA_T}\frac{\kp}{2}\big(|u_+|^2-|u_-|^2\big)\dd \mathcal{H},
\\
\dom\frq_{\cA_{T}}&:=\{u\in H^1(\cA_T\setminus\Sigma_+; \mathbb{C}^2)\colon u(\cdot,2T)=0,\ u_-|_{\Sigma_{+}}=M_\tau u_+|_{\Sigma_{+}}\},
\\
\frq_{\cB_T}[u]&:=\int_{\cB_T}\big(|\nabla u |^2-V_T|u|^2\big)\dd x+\frac{2m}{\tau}\int_{\Sigma_+\cap \cB_T}|u_+-u_-|^2\dd \mathcal{H},
\\
\dom\frq_{\cB_{T}}&:=\{u\in H^1(\cB_T\setminus\Sigma_+; \mathbb{C}^2)\colon u(\cdot,T)=0,\ u_-|_{\Sigma_{+}}=M_\tau u_+|_{\Sigma_{+}}\}.
\end{split}
\end{align*}
For $u\in\dom\fra_{\rm N}^+$ it follows that $\chi_{0,T} u\in\dom\frq_{\cA_{T}}$ and $\chi_{1,T} u\in\dom\frq_{\cB_{T}}$, so in particular one has for any $n\in\NN$
\begin{align*}
\begin{split}
\lambda_n(\fra^+_{\rm N})&=\inf_{\substack{V\subset\dom\fra^+_{\rm N}\\ \mathrm{dim}V=n}}\sup_{\substack{u\in V\\u\neq 0}}\frac{\fra_{\rm N}^+[u]}{\|u\|_{L^2(\R^2_+)}^2}
\\
&=\inf_{\substack{V\subset\dom\fra^+_{\rm N}\\ \mathrm{dim}V=n}}\sup_{\substack{u\in V\\u\neq 0}}\frac{\frq_{\cA_T}[\chi_{0,T} u]+\frq_{\cB_T}[\chi_{1,T} u]}{\|\chi_{0,T} u\|_{L^2(\cA_T)}^2+\|\chi_{1,T} u\|_{L^2(\cB_T)}^2}
\\
&=\inf_{\substack{V\subset\{(\chi_{0,T}u,\chi_{1,T}u)\colon u\in\dom\fra^+_{\rm N}\}\\ \mathrm{dim}V=n}}\sup_{\substack{(u_0,u_1)\in V\\(u_0,u_1)\neq 0}}\frac{\frq_{\cA_T}[u_0]+\frq_{\cB_T}[u_1]}{\|u_0\|_{L^2(\cA_T)}^2+\|u_1\|_{L^2(\cB_T)}^2}
\\
&\ge\inf_{\substack{V\subset\dom\frq_{\cA_T}\oplus\dom\frq_{\cB_T}\\ \mathrm{dim}V=n}}\sup_{\substack{(u_0,u_1)\in V\\(u_0,u_1)\neq 0}}\frac{\frq_{\cA_T}[u_0]+\frq_{\cB_T}[u_1]}{\|u_0\|_{L^2(\cA_T)}^2+\|u_1\|_{L^2(\cB_T)}^2}=\lambda_n(\frq_{\cA_T}\oplus \frq_{\cB_T}).
\end{split}
\end{align*}
But this again implies for any $\lambda\in\RR$
\begin{equation*}
N(\fra^+_{\rm N},\lambda)\le N(\frq_{\cA_T}\oplus \frq_{\cB_T},\lambda)=N(\frq_{\cA_T},\lambda)+N(\frq_{\cB_T},\lambda).
\end{equation*}
We analyze $\frq_{\cA_T}$ first and partition it once again. For that purpose we set 
\[
\begin{aligned}
	\cA_{0,T}:=\cA_T\cap\{x_1&<-2T\tan\omega\},\qquad 
	\cA_{1,T}:=\cA_T\cap\{|x_1| < 2T\tan\omega\},\\
	&\cA_{2,T}:=\cA_{T}\cap\{x_1>2T\tan\omega\}.
\end{aligned}	
\]
	\begin{figure}[h]
\includegraphics[width=12cm, keepaspectratio]{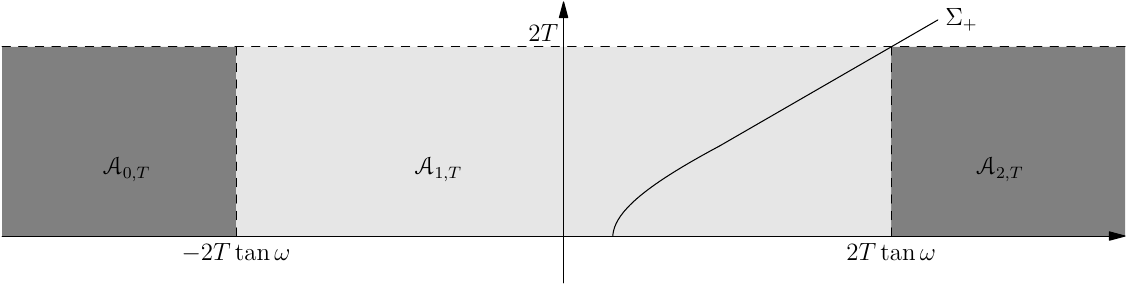}
	\caption{Schematic picture of the sets $\cA_{0,T}$, $\cA_{1,T}$, and $\cA_{2,T}$. The regions $\cA_{0,T}$ and $\cA_{2,T}$ are depicted in gray and $\cA_{1,T}$ in light gray.}
	\end{figure}
We choose $T > 0$ large enough so that $\cA_{0,T}\cap\Sg_+ =\varnothing$ and $\cA_{2,T}\cap\Sg_+ = \varnothing$ 
and introduce the quadratic forms
\begin{align*}
\frq_{\cA_{j,T}}[u]&:=\int_{\cA_{j,T}}\big(|\nabla u|^2-V_T|u|^2\big)\dd x,\quad \dom\frq_{\cA_{j,T}}:=\{u\in H^1(\cA_{j,T}; \mathbb{C}^2)\colon u(\cdot,2T)=0\},
\end{align*}
$ j\in\{0,2\}$, and
\[
\begin{aligned}
&\frq_{\cA_{1,T}}[u]:=\int_{\cA_{1,T}}\big(|\nabla u|^2-V_T|u|^2\big)\dd x\!+\!\frac{2m}{\tau}\int_{\Sigma_+\cap \cA_T}|u_+-u_-|^2\dd \mathcal{H}\! +\!\int_{\Sg_+\cap\cA_T}\frac{\kp}{2}\big(|u_+|^2-|u_-|^2\big)\dd \mathcal{H},\\
&\dom\frq_{\cA_{1,T}} :=\{u\in H^1(\cA_{1,T}\setminus\Sg_+; \mathbb{C}^2)\colon u(\cdot,2T)=0, u_-|_{\Sigma_{+}\cap\cA_T}=M_\tau u_+|_{\Sg_+\cap\cA_T}\}.
\end{aligned}
\]
By once again applying the min-max principle we have for any $n\in\NN$
that $\lambda_n(\frq_{\cA_T})\ge\lambda_n(\frq_{\cA_{0,T}}\oplus \frq_{\cA_{1,T}}\oplus \frq_{A_{2,T}})$. Hence for any $\lambda\in\RR$ we get
\begin{equation*}
N(\frq_{\cA_T},\lambda)\le N(\frq_{\cA_{0,T}}\oplus \frq_{\cA_{1,T}}\oplus \frq_{\cA_{2,T}},\lambda)=N(\frq_{\cA_{0,T}},\lambda)+N(\frq_{\cA_{1,T}},\lambda)+N(\frq_{\cA_{2,T}},\lambda).
\end{equation*}

\noindent \emph{Step 2:} We show the finiteness of $N(\frq_{\cA_T},\lm_\tau)$.
First, as $\frq_{\cA_{1,T}}$ is semibounded from below and closed, one has
\begin{equation*}
N(\frq_{\cA_{1,T}},\lambda)<\infty
\end{equation*}
for any $\lambda\in\RR$, since the operator associated to $\frq_{\cA_{1,T}}$ has compact resolvent and is semibounded from below. Also one can easily see that 
\begin{equation*}
\frq_{\cA_{0,T}}=\frt_{\rm N}\otimes \fri+ \fri\otimes \frt_{\rm DN},
\end{equation*}
where $\fri$ is the form of an identity operator, $\frt_{\rm N}$ corresponds to the one-dimensional Neumann Laplacian on the interval $(-\infty,-2T\tan\omg)$ and $\frt_{\rm DN}[u]:=\int_0^{2T}(|u'|^2-V_T(0,t)|u|^2)\dd t$ with $\dom\frt_{\rm DN}:=\{u\in H^1(0,2T):u(2T)=0\}$ corresponding to a one-dimensional Schr\"odinger operator on the interval $(0,2T)$ with Neumann boundary condition at $t = 0$ and Dirichlet boundary condition at $t = 2T$; here it is also used that $V_T$ is independent of $x_1$. In particular, we get $\frt_{\rm N}\ge 0$ and $\frt_{\rm DN}\ge -\|V_T\|_\infty$. Using \eqref{potential}, one can see that $\|V_T\|_\infty\le C/T^2$, with $C :=\|\chi_0'\|_\infty^2+\|\chi_1'\|_\infty^2$. Note that $\lambda_\tau < 0$. Thus we can choose $T$ big enough such that $-\|V_T\|_\infty>\lambda_\tau$. By applying the min-max principle to $\frq_{\cA_{0,T}}$, we then get 
\begin{equation*}
 N(\frq_{\cA_{0,T}},\lambda_\tau)=0.
\end{equation*}
The last quadratic form can be written as $\frq_{\cA_{2,T}}=\widetilde{\frt}_{\rm N}\otimes \fri+ 
\fri\otimes \frt_{\rm DN}$, where $\widetilde{\frt}_{\rm N}$ is the one-dimensional Neumann-Laplacian on the interval $(2T\tan\omega,\infty)$, \ie $\widetilde{\frt}_{\rm N}\ge 0$. Arguing as for the last quadratic form, we also get 
\begin{equation*}
N(\frq_{\cA_{2,T}},\lambda_\tau)=0.
\end{equation*}
All in all we get 
\begin{equation}\label{N0}
N(\frq_{\cA_T},\lambda_\tau)\le N(\frq_{\cA_{0,T}}\oplus \frq_{\cA_{1,T}}\oplus \frq_{\cA_{2,T}},\lambda_\tau)<\infty.
\end{equation}

\noindent \emph{Step 3:}
Now we want to show the finiteness of $N(\frq_{\cB_T},\lambda_\tau)$. For that purpose we
introduce the line $\widetilde{\Gamma}:=\{(x_1,x_2)\in\RR^2:x_2=\cot(\omega)x_1\}$, which splits the plane into two domains $\widetilde{\Omega}_\pm$ with common boundary $\widetilde{\Gamma}$ satisfying $\widetilde{\Omega}_\pm \cap \cB_T = \Omega_\pm \cap \cB_T$. In the following, we write for $u\in H^1(\RR^2\setminus\widetilde{\Gamma}; \mathbb{C}^2)$ also $u_\pm = u|_{\widetilde{\Omega}_\pm}$.
Define the quadratic form
\[
\begin{aligned}
\widetilde{\frq}_T[u]&:=\int_{\R^2}\big(|\nabla u|^2-V_T|u|^2\big)\dd x+\frac{2m}{\tau}\int_{\widetilde{\Gamma}}|u_+-u_-|^2\dd \mathcal{H},\\
\dom\widetilde{\frq}_T&:=\Bigl\{u\in H^1(\RR^2\setminus\widetilde{\Gamma}; \mathbb{C}^2)\colon u_-|_{\widetilde{\Gamma}}=\widetilde M u_+|_{\widetilde{\Gamma}}\Bigr\},\\
&\text{where}\quad\widetilde M := \frac{1}{1-\frac{1}{4}\tau^2}\begin{pmatrix}
	1+\frac14\tau^2 & -\ii e^{\ii\omega}\tau\\
	\ii e^{-\ii\omega}\tau & 1+\frac14\tau^2
	\end{pmatrix},
\end{aligned}
\]
where the potential $V_T$ in~\eqref{potential} is extended by zero to the lower half-plane. Let $u\in \dom\frq_{\cB_T}$.
Then one has $\widetilde{\frq}_T[\widetilde u]=\frq_{\cB_T}[u]$, where $\widetilde u$ is an extension of $u$ by $0$ to $\RR^2$. Therefore, by the min-max principle $\lambda_n(\frq_{\cB_T})\ge\lambda_n(\widetilde \frq_T)$ for any $n\in\NN$, and hence 
\begin{equation}\label{N1}
N(\frq_{\cB_T},\lambda_\tau)\le N(\widetilde \frq_T,\lambda_\tau).
\end{equation}
For convenience, we make a rotation so that the curve integral in $\widetilde{\frq}_T$ is an integral over the $x_2$-axis. Recall that we use the notation
\begin{equation*}
  \Gamma_0 = \{ (0, x_2): x_2 \in \mathbb{R} \}, \quad \check{\Omega}_\pm := \{ (x_1,x_2)^\top \in \mathbb{R}^2: \pm x_1 \geq 0 \},
\end{equation*}
and we write for $u \in H^1(\mathbb{R}^2 \setminus \Gamma_0; \mathbb{C}^d)$, $d \in \mathbb{N}$,
\begin{equation*}
  u_\pm := u |_{\check{\Omega}_\pm}.
\end{equation*}
Define the form
\[
\begin{aligned}
\frq_T[u]&:=\int_{\RR^2}\big(|\nabla u|^2-\widehat V_T |u|^2\big)\dd x+\frac{2m}{\tau}\int_{\RR}|u_+(0,x_2)-u_-(0,x_2)|^2\dd x_2,
\\
\dom\frq_T&:=\big\{u\in H^1(\RR^2 \setminus \Gamma_0; \mathbb{C}^2)\colon u_-|_{\Gamma_0}=\widetilde M u_+|_{\Gamma_0}\big\},
\end{aligned}
\]
where the potential $\widehat V_T$ is defined by
\begin{equation}\label{eq:whVT}
\widehat V_T(x_1,x_2):=V_T(x_1\cos(\omega)+x_2\sin(\omega), x_2\cos(\omega)-x_1\sin(\omega)).
\end{equation}
A direct calculation shows that $\frq_T[u]=\widetilde{\frq}_T[u\circ O]$ for any $u\in\dom\frq_T$, where $O$ is a rotation by $\omg$ in $\RR^2$ in the counterclockwise direction.
Since rotations induce unitary transformations, we have 
\begin{equation}\label{N2}
N(\frq_T,\lambda_\tau)=N(\widetilde \frq_{T},\lambda_\tau).
\end{equation}
Next, we rewrite $\frq_T$ as the sum of two forms acting on the components $u_1$ and $u_2$ of $u = (u_1,u_2)^\top \in \dom \frq_T$. For that purpose, we define the unitary matrix
\begin{equation*}
\Theta=\frac{1}{\sqrt{2}} \begin{pmatrix} 1 & \ii e^{\ii \omega} \\ \ii e^{-\ii \omega} & 1 \end{pmatrix},
\end{equation*}
inducing a unitary operator in $L^2(\mathbb{R}^2; \mathbb{C}^2)$ defined by $u\mapsto \Theta u$ that is again denoted by $\Theta$. By a direct computation one gets that
\begin{equation} \label{M_diagonalization}
  \Theta^* \widetilde{M} \Theta = \begin{pmatrix}
\frac{2+\tau}{2-\tau} &0\\ 0 &\frac{2-\tau}{2+\tau}
\end{pmatrix} =: M.
\end{equation}
Define for $\mu \in \mathbb{R} \setminus \{0, \pm 2 \}$ in $L^2(\mathbb{R}^2; \mathbb{C})$ the quadratic form
\begin{equation*}
  \begin{split}
    \check{\frq}_{T, \mu}[u] &= \int_{\RR^2}\big(|\nabla u|^2-\widehat V_T |u|^2\big)\dd x-\frac{2m}{|\mu|}\int_{\RR}|u_+(0,x_2)-u_-(0,x_2)|^2\dd x_2, \\
    \dom \check{\frq}_{T, \mu} &= \left\{ u \in H^1(\mathbb{R}^2 \setminus \Gamma_0; \mathbb{C}): u_-|_{\Gamma_0} = \frac{2-\mu}{2+\mu} u_+|_{\Gamma_0} \right\}.
  \end{split}
\end{equation*}
Taking~\eqref{M_diagonalization} into account, one finds for $u = (u_1, u_2)^\top \in L^2(\mathbb{R}^2; \mathbb{C}^2)$ that $\Theta u \in \dom \frq_T$ if and only if $u_1 \in \dom \check{\frq}_{T, -\tau}$ and $u_2 \in \dom \check{\frq}_{T, \tau}$ and in this case
\begin{equation*}
  \frq_T[\Theta u] = \check{\frq}_{T, -\tau}[u_1] + \check{\frq}_{T, \tau}[u_2].
\end{equation*}
Since $\Theta$ is unitary, this yields
\begin{equation}\label{N3}
N(\frq_T,\lambda_\tau)=N(\check{\frq}_{T, \tau},\lambda_\tau) + N(\check{\frq}_{T, -\tau},\lambda_\tau).
\end{equation}
We are going to prove that $N(\check{\frq}_{T, \tau},\lambda_\tau) < \infty$, the statement $N(\check{\frq}_{T, -\tau},\lambda_\tau) < \infty$ can be shown in the same way.

Recall that $\Pi_\tau$ and $\mathbb{P}_\tau$ are defined by~\eqref{def_projections1} and~\eqref{def_projections2}, respectively. 
By Lemma~\ref{lemma_projections}~(i) one has for any $u\in\dom\check\frq_{T, \tau}$ that $\Pi_\tau u \in \dom\check\frq_{T, \tau}$, implying that also $\mathbb{P}_\tau u \in \dom\check\frq_{T, \tau}$. Moreover,
\begin{equation}\label{proj}
\begin{aligned}
\check{\frq}_{T, \tau}[u]&=\check{\frq}_{T, \tau}[\Pi_\tau u + \mathbb{P}_\tau u]
=\check{\frq}_{T, \tau}[\Pi_\tau u]+\check{\frq}_{T, \tau}[\mathbb{P}_\tau u]-2 \Re\left(\int_{\RR^2}\widehat{V}_T \Pi_\tau u \cdot  \overline{\mathbb{P}_\tau u} \dd x\right),
\end{aligned}
\end{equation}
where we used, with $\Phi_\tau$ and $\psi_{\tau,u}$ defined by~\eqref{def_Phi_tau} and~\eqref{def_psi_tau}, respectively,
\begin{align*}
\begin{split}
\int_{\RR^2}& \nabla(\Pi_\tau  u) \cdot \overline{\nabla(\mathbb{P}_\tau u)} \dd x+\frac{2m}{\tau}\int_{\RR} ((\Pi_\tau  u)_+-(\Pi_\tau  u)_-)(0,x_2) \cdot \overline{((\mathbb{P}_\tau u)_+-(\mathbb{P}_\tau u)_-)(0,x_2)}\dd x_2
\\
=&\int_{\RR}\psi_{\tau,u}(x_2)\left[\int_{\RR} \Phi_\tau' \cdot \overline{\partial_1\mathbb{P}_\tau u} \dd x_1+\frac{2m}{\tau} \big(\Phi_\tau(0^+)-\Phi_{\tau}(0^-)\big) \cdot \overline{\big((\mathbb{P}_\tau u)_+-(\mathbb{P}_\tau u)_-\big)(0,x_2)} \right] \dd x_2
\\
&\quad+\int_{\RR^2} \Phi_\tau(x_1) \psi_{\tau,u}'(x_2) \cdot \overline{\partial_2\mathbb{P}_\tau u(x)} \dd x
\\
=&\int_{\RR}E_1(\sfP_\tau) \int_{\mathbb{R}} \Phi_\tau(x_1) \psi_{\tau,u}(x_2) \cdot \overline{\mathbb{P}_\tau u(x_1,x_2)} \dd x_1 \dd x_2 + \int_{\RR^2} \Phi_\tau(x_1) \psi_{\tau,u}'(x_2) \cdot \overline{\partial_2\mathbb{P}_\tau u(x)} \dd x
\\
=&\int_{\RR^2}\Big(\Phi_\tau(x_1) \psi_{\tau,u}'(x_2) \cdot \overline{\partial_2 u}-|\Phi_\tau(x_1) \psi_{\tau,u}'(x_2)|^2\Big)\dd x
\\
=&\int_{\RR}\psi_{\tau,u}'(x_2)\partial_2\left(\int_{\RR} \Phi_\tau(x_1)  \cdot \overline{u(x_1,x_2)} \dd x_1\right) \dd x_2
-\|\psi_{\tau,u}'\|^2_{L^2(\RR; \CC)}=0;
\end{split}
\end{align*}
here, the second equality holds by Lemma~\ref{lemma_1D}, the third one by the orthogonality of $\Phi_\tau \psi_{\tau,u}(x_2)$ and $\mathbb{P}_\tau u(\cdot,x_2)$ from Lemma~\ref{lemma_projections}~(ii), and the definition of $\mathbb{P}_\tau$, and the fourth equality by the fact that $\Phi_\tau$ has norm $1$.
The first summand in \eqref{proj} can be written as
\begin{align*}
\check{\frq}_{T, \tau}[\Pi_\tau u]=
&\int_{\RR^2}
\Big(|\Phi_\tau'(x_1)\psi_{\tau,u}(x_2)|^2
+
|\Phi_\tau(x_1)\psi_{\tau,u}'(x_2)|^2\Big)\dd x
\\
&+\frac{2m}{\tau}
\int_{\RR}|(\Phi_\tau(0^+)-\Phi_\tau(0^-))\psi_{\tau,u}(x_2)|^2\dd x_2-
\int_{\RR^2}\widehat{V}_T|\Pi_\tau u|^2\dd x
\\
=&\int_{\RR} |\psi_{\tau,u}(x_2)|^2\left[\int_{\RR} |\Phi_\tau'(x_1)|^2\dd x_1+\frac{2m}{\tau}|\Phi_\tau(0^+)-\Phi_\tau(0^-)|^2\right]\dd x_2
\\
&+\int_{\RR}|\psi_{\tau,u}'(x_2)|^2\dd x_2-\int_{\RR}|\psi_{\tau,u}(x_2)|^2\int_{\RR}\widehat{V}_T|\Phi_\tau(x_1)|^2\dd x_1\dd x_2
\\
=&\int_{\RR}\left(|\psi_{\tau,u}'(x_2)|^2+E_1(\sfP_\tau)|
\psi_{\tau,u}(x_2)|^2-|\psi_{\tau,u}(x_2)|^2\int_{\RR}\widehat{V}_T|\Phi_\tau(x_1)|^2\dd x_1\right)\dd x_2,
\end{align*}
where in the last step it was used that $\Phi_\tau$ is a normalized eigenfunction of the operator $\sfP_\tau$ from Lemma~\ref{lemma_1D}.
Furthermore, using Lemma~\ref{lemma_projections}~(iii) and $\| \widehat{V}_T \|_\infty \leq C/T^2$ the second summand in \eqref{proj} can be estimated by
\begin{align*}
\check{\frq}_{T, \tau}[\mathbb{P}_\tau u]=&\int_{\RR^2}|\partial_2 \mathbb{P}_\tau u|^2\dd x-\int_{\RR^2}\widehat{V}_T|\mathbb{P}_\tau u|^2\dd x
\\
&+\int_{\RR}\left[\int_{\RR}|\partial_1 (\mathbb{P}_\tau u)(x_1,x_2)|^2\dd x_1+\frac{2m}{\tau}|(\mathbb{P}_\tau u)(0^+,x_2)-(\mathbb{P}_\tau u)(0^-,x_2)|^2\right]\dd x_2
\\
\ge&-\frac{C}{T^2}\|\mathbb{P}_\tau u\|^2_{L^2(\RR^2; \CC)}.
\end{align*}
To estimate the last summand in \eqref{proj} we use a Young-type inequality, namely $2ab\le\eps a^2+b^2/\eps$ for any $a,b\in\RR$ and any $\eps>0$. So in particular
\begin{equation*}
2\left|\int_{\RR^2} \widehat{V}_T \Pi_\tau u \cdot \overline{\mathbb{P}_\tau u}\dd x\right|\le\eps\|\mathbb{P}_\tau u\|^2_{L^2(\RR^2; \CC)}+\frac{1}{\eps}\|\widehat{V}_T\Pi_\tau u\|^2_{L^2(\RR^2; \CC)}.
\end{equation*}
Combining the above calculations we get
\begin{align*}
\check{\frq}_{T, \tau}[u]\ge&\int_{\RR}\left(|\psi_{\tau,u}'(x_2)|^2+E_1(\sfP_\tau)|
\psi_{\tau,u}(x_2)|^2-|\psi_{\tau,u}(x_2)|^2\int_{\RR}\widehat{V}_T(x_1, x_2)|\Phi_\tau(x_1)|^2\dd x_1\right)\dd x_2
\\
&-\frac{C}{T^2}\|\mathbb{P}_\tau u\|^2_{L^2(\RR^2; \CC)}-\eps \|\mathbb{P}_\tau u\|^2_{L^2(\RR^2; \CC)}-\frac{1}{\eps}\|\widehat{V}_T\Pi_\tau u\|^2_{L^2(\RR^2; \CC)}
\\
=&\int_{\RR}\left(|\psi_{\tau,u}'(x_2)|^2-|\psi_{\tau,u}(x_2)|^2\int_{\RR}|\Phi_\tau(x_1)|^2\left(\widehat{V}_T(x_1,x_2)+\frac{1}{\eps}|\widehat{V}_T(x_1,x_2)|^2\right)\dd x_1\right)\dd x_2
\\
&+ E_1(\sfP_\tau)\|u\|^2_{L^2(\RR^2; \CC)} - \left(\eps+\frac{C}{T^2}+ E_1(\sfP_\tau) \right) \| \mathbb{P}_\tau u \|^2_{L^2(\RR^2; \CC)}
\\
\ge&\int_{\RR}\left(|\psi_{\tau,u}'(x_2)|^2-|\psi_{\tau,u}(x_2)|^2\int_{\RR}|\Phi_\tau(x_1)|^2\left(\widehat{V}_T(x_1,x_2)+\frac{1}{\eps}|\widehat{V}_T(x_1,x_2)|^2\right)\dd x_1\right)\dd x_2
\\
&+ E_1(\sfP_\tau)\|u\|^2_{L^2(\RR^2; \CC)},
\end{align*} 
where we used $\|\Pi_\tau u\|^2_{L^2(\RR^2; \CC)}=\|\psi_{\tau,u}\|^2_{L^2(\RR; \CC)}$ and the identity  $\|\Pi_\tau u\|^2_{L^2(\RR^2; \CC)} + \|\mathbb{P}_\tau u\|^2_{L^2(\RR^2; \CC)} = \|u\|^2_{L^2(\RR^2; \CC)}$, and we also chose $\eps,T>0$ such that the inequality $\eps+C/T^2\le |E_1(\sfP_\tau)|$ holds. Next, define
\begin{align*}
Z_{T}(x_2):=&\int_{\RR}|\Phi_\tau(x_1)|^2\left(\widehat{V}_T(x_1,x_2)+\frac{1}{\eps}|\widehat{V}_T(x_1,x_2)|^2\right)\dd x_1,
\\
\frb_{T}[\psi]:=&\int_{\RR}\big(|\psi'|^2-Z_{T}|\psi|^2\big)\dd t,\qquad \dom\frb_{T}:=H^1(\RR; \CC). 
\end{align*}
As $E_1(\sfP_\tau)= \lambda_\tau$, one has
\begin{equation} \label{N35}
N(\check{\frq}_{T, \tau},\lambda_\tau)\le N(\frb_{T},0).
\end{equation}
According to a Bargmann type estimate, see~\cite[Theorem XIII.9 (a)]{RSIV} (for $\ell = 0$) and \cite[Equation (8)]{simon}, one has
\begin{equation}\label{N4}
N(\frb_{T},0)\le 2+\int_{\RR}|x_2| Z_{T}(x_2)\dd x_2,
\end{equation}
where we implicitly used that $Z_T\ge 0$ which follows from~\eqref{potential}
	combined with~\eqref{eq:whVT}.
So we need to show the finiteness of the latter integral. But this can be seen as follows
\begin{align*}
\int_{\RR}|x_2| Z_{T}(x_2)\dd x_2=&\int_{\RR}|x_2|
\int_{\RR}|\Phi_\tau(x_1)|^2\left(\widehat{V}_T(x_1,x_2)+\frac{1}{\eps}|\widehat{V}_T(x_1,x_2)|^2\right)\dd x_1\dd x_2
\\
=&\int_{\RR}|\Phi_\tau(x_1)|^2\int_{\RR}|x_2|\left(\widehat{V}_T(x_1,x_2)+\frac{|\widehat{V}_T(x_1,x_2)|^2}{\eps}\right)\dd x_2\dd x_1
\\
=&\int_{\RR}|\Phi_\tau(x_1)|^2\Bigg(\int_{\frac{T+x_1\sin\omega}{\cos\omega}}^{\frac{2T+x_1\sin\omega}{\cos\omega}}|x_2|\left(\widehat{V}_T(x_1,x_2)+\frac{1}{\eps}|\widehat{V}_T(x_1,x_2)|^2\right)\dd x_2  \Bigg) \dd x_1\\
&\le
\frac{1}{\cos \omg}
\left(\frac{C}{T^2}+\frac{C^2}{ T^4\eps}\right)	
\int_{\RR}|\Phi_\tau(x_1)|^2\big(2T+|x_1|\sin\omega\big)\left(\int_{\frac{T+x_1\sin\omg}{\cos\omg}}^{\frac{2T+x_1\sin\omg}{\cos \omg}}\dd x_2\right)\dd x_1
\\
\le&\frac{C}{T(\cos\omega)^2}\left(1+\frac{C}{T^2\eps}\right)\int_{\RR}|\Phi_\tau(x_1)|^2\big(2T+|x_1|\sin\omega\big)\dd x_1<\infty.
\end{align*}
For the last inequality we used that $|\Phi_\tau(x_1)|\leq C'\exp(-c|x_1|)$ with some constants $C',c > 0$. Now we can finally combine \eqref{N1}, \eqref{N2}, \eqref{N3}, \eqref{N35}, and \eqref{N4} to get
\begin{equation}\label{N5}
N(\frq_{\cB_T},\lambda_\tau)<\infty.
\end{equation}

\noindent\emph{Step 4:} In this step, we use the results from the previous steps to conclude the claims of the theorem.
Since $N(\frq_{\cB_T},\lambda_\tau)<\infty$ and $N(\frq_{\cA_T},\lambda_\tau)<\infty$ according to \eqref{N0} and \eqref{N5} we also have 
\begin{equation*}
	N(\fra_{\rm N}^+,\lambda_\tau)<\infty.
\end{equation*}
To show the finiteness of $N(\fra_{\rm N}^-,\lambda_\tau)$ one follows the same steps as for $\fra_{\rm N}^+$. Therefore, by \eqref{N-1} one has $N(\fra_{\rm N},\lambda_\tau)<\infty$ and by \eqref{N-2} one concludes
\begin{equation*}
	N(\fra,\lambda_\tau)<\infty.
\end{equation*}
Since $\frd_{\tau,m}$ is $\fra$ shifted by $m^2$ we have
\begin{equation*}
N(\frd_{\tau,m},\lambda_\tau+m^2)<\infty.
\end{equation*}
Also remark that 
\begin{equation*}
\lambda_\tau+m^2= m^2\left(\frac{\tau^2-4}{\tau^2+4}\right)^2.
\end{equation*}
By combining the last two displayed formulas, one gets~\eqref{form_finite}.
Now by the definition of $\frd_{\tau,m}$ we have $\frd_{\tau,m}[u]=\|\sfD_{\tau,m}u\|^2_{L^2(\RR^2)}$
and the domains of the self-adjoint operator $\sfD_{\tau,m}$ and
the quadratic form $\frd_{\tau,m}$ coincide. Thus, $\sfD_{\tau,m}^2$ represents the form $\frd_{\tau,m}$ by the first representation theorem. In view of Proposition~\ref{essentials} combined with the spectral theorem
the number of eigenvalues of the operator $\sfD_{\tau,m}$ in the gap of its essential spectrum with multiplicities taken into account coincides with $N(\frd_{\tau,m},\lm_\tau+m^2)$ and thus, this number is finite.
\end{proof}

\section{Existence of eigenvalues in the gap}
\label{sec:existence}
In this section, we provide a sufficient condition for Dirac operators with a Lorentz scalar $\dl$-shell interaction supported 
on a class of convex curves as in Hypothesis~\ref{hyp} to have non-empty discrete spectrum in the gap of the essential spectrum. For that purpose, we compute the representation of the quadratic form in Proposition~\ref{essentials} in parallel coordinates. In the following, we suppress the target space $\CC^2$ in the notation for the $L^2$-spaces. The coordinate transformation $\cU$ is understood to act componentwise on $u=(u_1,u_2)^\top$. We introduce the unitary transformation
\[
\cU\colon L^2(\dR^2)\to L^2(U;(1-\kp(s)t)\dd s\dd t),\qquad \mathcal{U}u:= u\circ\Phi,
\]
where we used that $|\det \sfJ_\Phi(s,t)| = 1-\kp(s)t$ for all $(s,t)\in U$, which follows from~\eqref{eq:detJacobian}
combined with the definition of $U$ in~\eqref{eq:U} and the bound~\eqref{eq:boundc(s)}. We also implicitly use in the definition of $\cU$ that the mapping $\Phi$ in~\eqref{eq:Phi} is a diffeomorphism from $U$ onto $\dR^2\sm{\rm Cut}\,(\Sg)$, where ${\rm Cut}\,(\Sg)$ has zero Lebesgue measure.
In the next lemma we compute how the quadratic form $\frd_{\tau,m}$ in Proposition\,\ref{essentials}\,(iii) changes under the unitary transform $\cU$. 

\begin{lem}\label{shifted quad}
	Let $\Sg\subset\dR^2$ be a curve as in Hypothesis~\ref{hyp}. Then, for all $u\in \dom\sfD_{\tau,m}$ we have
	\begin{equation}\label{quad trafo}
		\frd_{\tau,m}[u]-m^2\left(\frac{\tau^2-4}{\tau^2+4}\right)^2\|u\|^2_{L^2(\RR^2,\CC^2)}= Q_1[v]+Q_2[v]=:Q[v]
	\end{equation}
	with $v:= \mathcal{U}u$ and
	\[
	\begin{aligned}
		Q_1[v] &:= \int_\dR
		\int_{-\infty}^{c(s)}\frac{|\p_s v|^2}{1-\kp(s)t}\dd t 
		\dd s,\\
		Q_2[v] &:= \int_\dR\bigg[
		\int_{-\infty}^{c(s)}\left(|\p_t v|^2(1-\kp(s)t)
		+\frac{16m^2\tau^2}{(\tau^2+4)^2}|v(s,t)|^2(1-\kp(s)t)\right)\dd t\\
		&\qquad \qquad + \frac{\kp(s)}{2}\big(|v(s,0^+)|^2-|v(s,0^-)|^2\big)+
		\frac{2m}{\tau}|v(s,0^+)-v(s,0^-)|^2\bigg]\dd s.
	\end{aligned}
	\]
\end{lem}
\begin{proof}
	We are going to represent the terms $\int_{\dR^2}|u|^2\dd x$ and $\int_{\dR^2}|\nb u|^2\dd x$ as well as the boundary terms in~\eqref{equation_form_square} in parallel coordinates. Changing the variables in the integral, we get
	\[
	\int_{\R^2}|u|^2\dd x=\int_U |(u\circ\Phi)(s,t)|^2 (1-\kappa(s) t)\dd s\dd t=\int_U |v(s,t)|^2 (1-\kappa(s) t)\dd s\dd t.
	\]
	Next, using $|\gamma'(s)| = 1$ for all $s \in \mathbb{R}$, we deduce from the formula for the Jacobian~\eqref{eq:J} the following auxiliary identity:
	\[
	\begin{aligned}
		&\sfJ_\Phi(s,t)^{-1}\big(\sfJ_\Phi(s,t)^{-1}\big)^\top\\ &\quad = 
		\frac{1}{(1-\kp(s)t)^2}
		\begin{pmatrix}  -\gamma_1'(s) & 
			-\gamma_2'(s)\\
			-\gamma_2'(s)+t\kp(s)\gamma_2'(s)
			&\gamma_1'(s)-t\kp(s)\gamma_1'(s)
		\end{pmatrix}
		\begin{pmatrix}  -\gamma_1'(s) & 
			-\gamma_2'(s)+t\kp(s)\gamma_2'(s)\\
			-\gamma_2'(s)
			&\gamma_1'(s)-t\kp(s)\gamma_1'(s)
		\end{pmatrix}\\
		&\quad=\frac{1}{(1-\kp(s)t)^2}
		\begin{pmatrix} 1 & 0 \\
			0&(1-\kp(s)t)^2\end{pmatrix}.
	\end{aligned}	
	\]
	Using the above formula and the chain rule for the gradient, we get for the first component $u_1$ of $u$ 
	\begin{align*}
		\int_{\R^2}|\nabla u_1|^2\dd x&=\int_U |((\nabla u_1)\circ\Phi)(s,t)|^2 (1-\kappa(s)t)\dd s\dd t
		\\
		&=\int_U(\nabla(u_1\circ\Phi)(s,t))
		\cdot (\sfJ_\Phi(s,t))^{-1}((\sfJ_\Phi(s,t))^{-1})^\top \ov{(\nabla(u_1\circ\Phi)(s,t))}(1\!-\!\kappa(s)t)\dd s\dd t
		\\
		&=\int_U\frac{|\p_s v_1|^2}{1-\kp(s)t}\dd s \dd t+\int_U |\p_t v_1|^2(1-\kp(s)t) \dd s\dd t,
	\end{align*}
	where $v_1$ denotes the first component of $v$. Analogously, we get the same formula for the second component of $u$. Summing both identities yields
	\begin{equation*}
		\int_{\R^2}|\nabla u|^2\dd x 
		=\int_U\frac{|\p_s v|^2}{1-\kp(s)t}\dd s \dd t+\int_U |\p_t v|^2(1-\kp(s)t) \dd s\dd t.
	\end{equation*}
	For the boundary terms we remark that $\gamma$ is an arc-length
	parametrization of $\Sg$, therefore
	\begin{align*}
		\int_{\Sg}\kappa|u_+|^2\dd \mathcal{H}&=\int_\R\kappa(s) |u_+(\gamma(s))|^2\dd s=\int_\R \kappa(s)|(u_+\circ\Phi) (s,0)|^2 \dd s=\int_\R \kappa(s)|v(s,0^+)|^2\dd s.
	\end{align*}
	The remaining boundary terms transform analogously. Plugging the above results into the formula for $\frd_{\tau,m}$ in Proposition~\ref{essentials}\,(iii) yields \eqref{quad trafo}.
\end{proof}

Recall that, throughout this section, the mass $m > 0$ is fixed.
We are now ready to formulate and prove the second main result of the paper about the discrete spectrum of  the self-adjoint Dirac operator $\sfD_{\tau,m}$ with a Lorentz scalar $\dl$-shell interaction of strength $\tau < 0$
supported on $\Sg$ (defined as in~\eqref{eq:DiracOperator}). 
\begin{thm}\label{thm:disc}
	Let $\Sg\subset\dR^2$ be a curve as in Hypothesis~\ref{hyp}. Then there exists $\tau_\star = \tau_\star(\Sg,m) \in (-2,0)$ such that the discrete spectrum of $\sfD_{\tau,m}$ is non-empty for all $\tau\in (-\infty,\frac{4}{\tau_\star})\cup(\tau_\star,0)$.
\end{thm}
\begin{proof}
	In the proof, we will employ parallel coordinates on $\Phi(U)=\dR^2\sm{\rm Cut}\,(\Sg)$. Let $s_1,s_2\in\dR$, $s_1 < s_2$, be as in Proposition~\ref{assumption}.
	Let us consider the functions $v = v_+\oplus v_-\colon U\arr\dC^2$ defined by
	\[
	v_+(s,t) := e^{\frac{4m\tau}{\tau^2+4}t}\begin{pmatrix} 1\\ 0\end{pmatrix}\ \text{ if }t>0,\qquad v_-(s,t) := 
	e^{-\frac{4m\tau}{\tau^2+4}t}\begin{pmatrix} \frac{1+\frac14\tau^2}{1-\frac14\tau^2}\\ \frac{\ii(\nu_1(s)+\ii\nu_2(s))\tau}{1-\frac14\tau^2}\end{pmatrix}\ \text{ if }t<0,
	\]
	and the cut-off function $\varphi_n\colon\dR\arr[0,1]$, $n\in\dN$, given by
	\[
	\varphi_n(s) := \wt\varphi_n\left(s-\frac{s_1+s_2}{2}\right),\qquad\text{where}
	\quad\wt\varphi_n(s) = \begin{cases}
		1,&\qquad |s|< n,\\
		\frac{2n-|s|}{n},&\qquad n\le|s|\le 2n,\\
		0,&\qquad |s| > 2n.
	\end{cases}	
	\]
	We employ
	\[
	v_n(s,t) = \varphi_n(s)v(s,t),\qquad n\in\dN,
	\]
	as the family of trial functions for the transformed Dirac operator $\sfD_{\tau,m}$. We will argue at the end of the proof
	that $\cU^{-1}v_n\in\dom\sfD_{\tau,m}$ for all $n\in\dN$ sufficiently large. As in Lemma \ref{shifted quad}, we introduce the sequence of real numbers $Q_j[v_n]$, $j=1,2$; formally, $Q_1[v_n] + Q_2[v_n]$ coincides with the shifted quadratic form of the square of the Dirac operator $\sfD_{\tau,m}$ applied to $v_n$.
	To find the asymptotics of $Q_1[v_n]$, we use that
	the Frenet formula~\eqref{Frenet} and $|\nu| = 1$ imply
	\begin{equation*}
		|\nu_1'(s) + \ii \nu_2'(s)| = |\gamma_2''(s) - \ii \gamma_1''(s)| = |\kappa(s)| \cdot |\nu_2(s) - \ii \nu_1(s)| = |\kappa(s)|.
	\end{equation*}
	Since $\kappa(s) = 0$ for all $s$ in the support of $\varphi_n'$, if $n$ is sufficiently large, this implies for $Q_1[v_n]$ the following asymptotic bound:
	\begin{equation}\label{eq:Q1}
		\begin{aligned}
			Q_1[v_n]&= \int_\dR\int_{-\infty}^{c(s)}|\varphi_n'(s)|^2|v(s,t)|^2\dd t\dd s + \int_{s_1}^{s_2}\int_{-\infty}^0
			\frac{\kp^2(s)\tau^2 e^{-\frac{8m\tau}{\tau^2+4} t}}{(1-\frac14\tau^2)^2(1-\kp(s) t)}\dd t\dd s\\
			&\le
			\frac{|\tau|}{8m}\frac{\tau^2+4}{(1-\frac14\tau^2)^2}
			\int_{s_1}^{s_2}\kp^2(s)\dd s+o(1)\qquad \text{ for }\ n\arr\infty.
		\end{aligned}
	\end{equation}
	Next, as $|\partial_t v(s,t)| = \frac{4 m |\tau|}{\tau^2 + 4} |v(s,t)|$, the term $Q_2[v_n]$ can be expressed as
	\[
	\begin{aligned}
		Q_2[v_n] & = \int_{\dR}|\varphi_n(s)|^2
		\Bigg[
		\frac{32m^2\tau^2}{(\tau^2+4)^2}\int_{-\infty}^{c(s)}|v(s,t)|^2(1-\kp(s)t)\dd t\\
		&\qquad\qquad\qquad + \frac{2m}{\tau}|v(s,0^+)-v(s,0^-)|^2 +\frac{\kp(s)}{2}|v(s,0^+)|^2 -\frac{\kp(s)}{2}|v(s,0^-)|^2\Bigg]\dd s.
	\end{aligned}
	\]
	The integral with respect to $t$ appearing in the above expression for $Q_2[v_n]$ can be written as
	\[
	\begin{aligned}
		I & := \int_{-\infty}^{c(s)}|v(s,t)|^2(1-\kp(s)t)\dd t \\
		&= \int_{-\infty}^0 \underbrace{\frac{\tau^2+(1+\frac14\tau^2)^2}{(1-\frac14\tau^2)^2}}_{=:a(\tau)}
		e^{-\frac{8m\tau t}{\tau^2+4}}(1-\kp(s)t)\dd t + \int_0^{c(s)} e^{\frac{8m\tau t}{\tau^2+4}}(1-\kp(s)t)\dd t
	\end{aligned}
	\]
	and further computed
	\[
	\begin{aligned}
		I &= a(\tau)\frac{\tau^2+4}{8m|\tau|} + a(\tau)\kp(s)\frac{(\tau^2+4)^2}{64m^2\tau^2} + \frac{\tau^2+4}{8m|\tau|}\Big[1-e^{\frac{8m\tau}{\tau^2+4}c(s)}\Big]\\ &\qquad\qquad\qquad\qquad-\kp(s)\frac{(\tau^2+4)^2}{64m^2\tau^2}\Big[1-\Big(1-\frac{8m\tau}{\tau^2+4}c(s)\Big)e^{\frac{8m\tau}{\tau^2+4}c(s)}\Big].
	\end{aligned}	
	\]
	As the next step, we evaluate the limit of $Q_2[v_n]$ as $n\arr\infty$ using the Lebesgue dominated convergence theorem and the explicit form of $c(s)$ for $s \notin [s_1, s_2]$ from Proposition~\ref{assumption}~(iii) as
	\begin{equation}\label{eq:Q2}
		\begin{aligned}
			Q_2[v_n] & \!=\!\int_\dR|\varphi_n(s)|^2\Bigg(a(\tau)\frac{4m|\tau|}{\tau^2+4} + a(\tau)\frac{\kp(s)}{2} + \frac{4m|\tau|}{\tau^2+4} -\frac{4m|\tau|}{\tau^2+4}e^{\frac{8m\tau}{\tau^2+4}c(s)}-\frac{\kp(s)}{2}\\
			&\quad  + \frac{\kp(s)}{2}\left(1-\frac{8m\tau}{\tau^2+4}c(s)\right)e^{\frac{8m\tau}{\tau^2+4}c(s)} + \frac{2m\tau(1+\frac14\tau^2)}{(1-\frac14\tau^2)^2} +\frac{\kp(s)}{2} -\frac{\kp(s)}{2}a(\tau)\Bigg)\dd s\\
			&
			\!=\!\int_\dR|\varphi_n(s)|^2\Bigg( -\frac{4m|\tau|}{\tau^2+4}e^{\frac{8m\tau}{\tau^2+4}c(s)}
			+ \frac{\kp(s)}{2}\left(1-\frac{8m\tau}{\tau^2+4}c(s)\right)e^{\frac{8m\tau}{\tau^2+4}c(s)} \Bigg)\dd s\\
			&\xrightarrow[n\arr\infty]{}
			I_1 + I_2 + I_3,
		\end{aligned}
	\end{equation}
	where the terms $\{I_j\}_{j=1}^3$ are given by
	\begin{equation*}
		\begin{aligned}
			I_1 &:= \frac{4m\tau}{\tau^2+4} \int_{s_1}^{s_2} e^{\frac{8m\tau}{\tau^2+4}c(s)}\dd s,\\
			I_2 &
			:=
			\frac{8m\tau}{\tau^2+4} \int_0^{\infty} e^{\frac{8m\tau}{\tau^2+4}\left(\frac{h}{\sin \omg} + s\cot\omg\right)}\dd s,\\
			I_3 & :=  \int_{s_1}^{s_2}\frac{\kp(s)}{2}
			\left(1-\frac{8m\tau}{\tau^2+4}c(s)\right)e^{\frac{8m\tau}{\tau^2+4}c(s)}\dd s.
		\end{aligned}
	\end{equation*}
	Using the bound~\eqref{eq:boundc(s)} in the first estimate
	for the sum $I_1+I_3$ we get
	\[
	I_1+I_3 = \frac{4m\tau}{\tau^2+4} \int_{s_1}^{s_2} e^{\frac{8m\tau}{\tau^2+4}c(s)} \left( 1 - \kappa(s) c(s)  \right)\dd s
	+ \int_{s_1}^{s_2}\frac{\kp(s)}{2}
	e^{\frac{8m\tau}{\tau^2+4}c(s)}\dd s
	\le \int_{s_1}^{s_2}\frac{\kp(s)}{2}\dd s = \omg.
	\]
	Moreover, the term $I_2$ can be computed exactly,
	\[
	I_2 = - e^{\frac{8m\tau}{\tau^2+4}\frac{h}{\sin\omega}}\tan\omg.
	\]
	Combining the limit of $Q_2[v_n]$ in~\eqref{eq:Q2} with the above estimates and with the asymptotic upper bound on $Q_1[v_n]$ in~\eqref{eq:Q1}, we arrive at
	\begin{equation} \label{eq:cond1}
	\limsup_{n\arr\infty}Q[v_n] \le 
	\omg -e^{\frac{8m\tau}{\tau^2+4}\frac{h}{\sin\omega}}\tan\omega
	+\frac{|\tau|}{8 m}\frac{\tau^2+4}{(1-\frac14\tau^2)^2}
	\int_{s_1}^{s_2}\kp^2(s)\dd s.
	\end{equation}
	Passing to the limits $\tau \arr 0^-$ and $\tau\arr-\infty$
	in this upper bound, we get
	\begin{equation}\label{eq:limitsbnds}
		\begin{aligned}
			\lim_{\tau\arr-\infty}\limsup_{n\arr\infty}Q[v_n] &\le \omg - \tan\omg  < 0,\\
			\lim_{\tau\arr0^-}\limsup_{n\arr\infty}Q[v_n] &\le \omg - \tan\omg  < 0,	
		\end{aligned}	
	\end{equation}
	where we used the elementary inequality $x < \tan x$ valid for all $x\in(0,\frac{\pi}{2})$. The claim of the theorem is a direct consequence of~\eqref{eq:limitsbnds}, if we can show that $u_n := \mathcal{U}^{-1}v_n\in \dom\sfD_{\tau,m}$. We remark that $u_n = v_n \circ \Phi^{-1}$, $\Phi^{-1}(\gamma(s))=(s,0)$, and $v_n(s,0^-)=M_\tau(s)v_n(s,0^+)$. The latter condition holding for all $s \in \mathbb{R}$ is equivalent to
	\[
(u_n)_- = M_\tau (u_n)_+ \text{ on } \Sigma.
	\]
	It is left to show that $u_n\in H^1(\dR^2\sm\Sg;\dC^2)$, but, due to~\eqref{quad trafo} and as we showed that $Q_j[v_n]<\infty$, $j=1,2$,  this immediately follows from the fact that $u_n$ is  continuous along the cut-locus for $n$ sufficiently large. This continuity holds because of~\eqref{symmetry_cut_locus}
	and because, by Proposition~\ref{assumption}\,(iii) and the construction of $v_n$ we have that $\Phi(s_1',c(s'_1)) = \Phi(s_2',c(s_2'))$ implies $\varphi_n(s_1') = \varphi_n(s_2')$ for $s_1',s_2'\in\dR$.

\end{proof}

	\begin{remark}
		The above proof yields, as a byproduct, that the condition
		\begin{equation}\label{eq:cond}
			\omg -e^{\frac{8m\tau}{\tau^2+4}\frac{h}{\sin\omega}}\tan\omega
			+\frac{|\tau|}{8 m}\frac{\tau^2+4}{(1-\frac14\tau^2)^2}
			\int_\Sg\kp^2\dd \mathcal{H} < 0
		\end{equation}
		is sufficient for the operator $\sfD_{\tau,m}$ with
		$\tau\in(-\infty,0)\sm\{-2\}$ and $m > 0$ to have discrete spectrum in the gap of the essential spectrum; \cf ~\eqref{eq:cond1}. This condition involves the coupling constant $\tau$, the mass $m$, and the geometric quantities $\omg$, $\int_\Sg\kp^2\dd \mathcal{H}$ and $h$. We expect that the minimal admissible value of the parameter $h$ can be estimated in terms of the curvature yielding a more explicit sufficient condition for the existence of the discrete eigenvalues. 
	\end{remark}	
	\begin{ex}
		\label{rem:cond}
		In this example, we derive from~\eqref{eq:cond} a more explicit sufficient condition for the existence of the discrete spectrum for $\sfD_{\tau,m}$ for a special class of curves and make plots of admissible values of the parameter $\tau$ in dependence on geometric parameters.
		
		Let $\rho \in C^\infty_0(\dR)$ be real-valued and such that
		\begin{multicols}{2}
			\begin{myenum}
				\item [{\rm (a)}] $0\le \rho \le 2$;
				\item [{\rm (b)}] $\supp\rho = [-1,1]$;
				\item [{\rm (c)}] $\int_{-1}^1\rho(x)\dd x = 2\omg$;  
				\item [{\rm (d)}] $\rho$ is an even function.
			\end{myenum}
		\end{multicols}
		Consider, for $R > 0$, the function
		\[
		\rho_R(x) = \frac{1}{R}\rho\left(\frac{x}{R}\right).
		\]
		In particular, we have $\supp \rho_R = [-R,R]$ and $\int_{-R}^R\rho_R(x) \dd x = 2\omg$. Let the curve $\Sg\subset\dR^2$ be such that its curvature satisfies $\kp = \rho_R$. It is straightforward to see that upon a suitable rotation and translation such a curve $\Sg$ is a local deformation of the broken line $\Sg_0$ defined in~\eqref{eq:broken}. Using standard formulas from differential geometry of curves in the plane~\cite[Theorem 1.88]{ASS}, we can reconstruct the arc-length parametrization of the curve $\Sg$
		in the clockwise orientation
		as
		\begin{equation}\label{eq:gammas}
			\begin{aligned}
				\gamma_1(s) &= \gamma_1(s_0) + \int_{s_0}^s\cos\left(\int_{s_0}^{s_1}\kp(t)\dd t\right)\dd s_1,\\
				\gamma_2(s) &= \gamma_2(s_0)+
				\int_{s_0}^s\sin\left(\int_{s_0}^{s_1}\kp(t)\dd t\right)\dd s_1.
			\end{aligned}	
		\end{equation}
		We can choose $s_0 = 0$ and assume without loss of generality that $ \gamma_2(0) = 0$. The second equation in~\eqref{eq:gammas} yields the bound
		\[	
		\gamma_2(R) \le R\sin\omg.
		\]
		Using the formula~\eqref{eq:hhprime} in the proof of Proposition~\ref{assumption} in Appendix~\ref{app} we derive that the admissible value of the parameter $h$ satisfies the bound
		\begin{equation}\label{eq:bndh}
			h \le R\sin\omg(1+\cos\omg).
		\end{equation}
		On the other hand, we also get
		\begin{equation}\label{eq:kp}
			\int_\Sg \kp^2(s)\dd \mathcal{H} =\frac{1}{R}\int_\dR\rho^2(x)\dd x \le \frac{2}{R}\int_\dR\rho(x)\dd x \le 
			\frac{4\omg}{R}.
		\end{equation}
		The condition~\eqref{eq:cond} and the bounds~\eqref{eq:bndh} and~\eqref{eq:kp} yield that the condition
	
		\begin{equation}\label{eq:cond2}
			\omg - e^{\frac{8m\tau}{\tau^2+4}R(1+\cos\omg)}\tan\omg + \frac{|\tau|}{8m}\frac{\tau^2+4}{(1-\frac14\tau^2)^2}\frac{4\omg}{R} < 0
		\end{equation}
		is sufficient for the operator $\sfD_{\tau,m}$ with
		the Lorentz-scalar $\dl$-shell interaction supported on $\Sg$ as above with the strength $\tau\in(-\infty,0)\sm\{-2\}$ and the mass $m >0$ to have discrete spectrum in the gap of the essential spectrum.
		Below we plot
		regions in which the existence of the discrete spectrum is guaranteed by the condition~\eqref{eq:cond2}.
		Namely, we plot for $m = R = 1$ such a region 
		in the $(\tau,\omg)$-plane in Figure~\ref{fig1}, further
		for $m = 1$ and $\omg = \frac{49\pi}{100}$ we plot in the $(\tau,R)$-plane in Figure~\ref{fig2}.
\end{ex}
\begin{figure}[H]
			\centering
			\begin{subfigure}{0.48\textwidth}
				\centering
				\includegraphics[width=\linewidth]{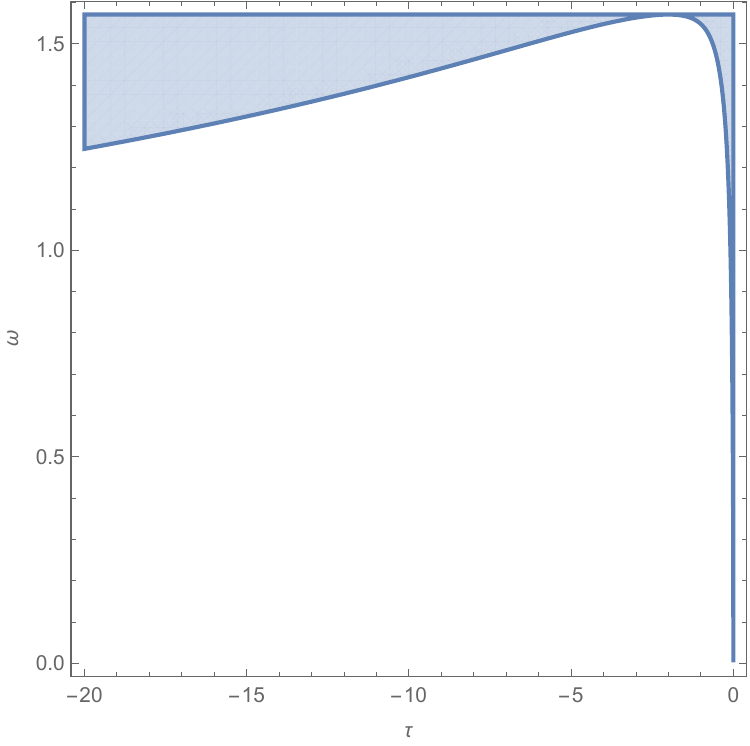}
				\caption{The region in $(\tau,\omg)$-plane with fixed $m = R = 1$.}
				\label{fig1}
			\end{subfigure}
			\hfill
			\begin{subfigure}{0.48\textwidth}
				\centering
				\includegraphics[width=\linewidth]{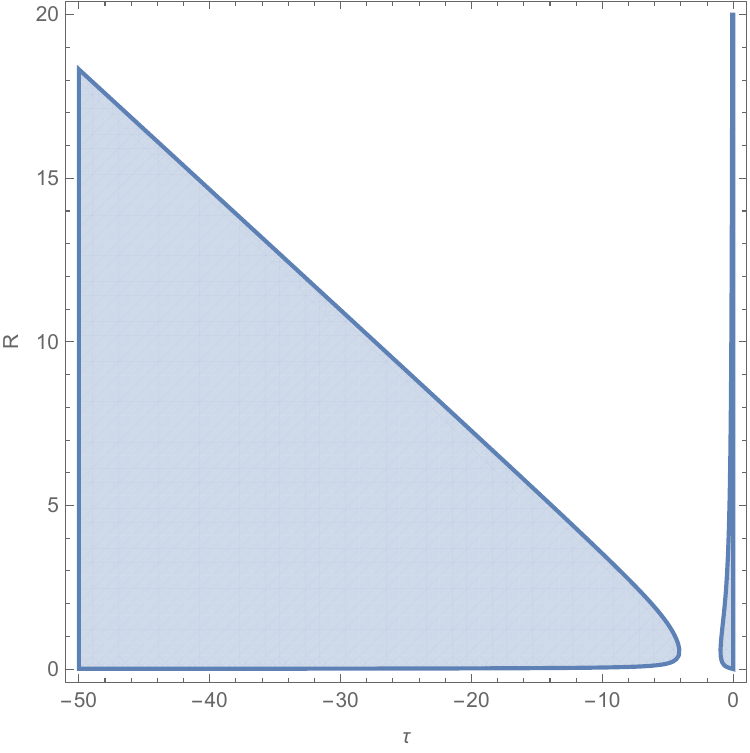}
				\caption{The region in $(\tau,R)$-plane with fixed $m = 1$ and $\omg = \frac{49\pi}{100}$.}
				\label{fig2}
			\end{subfigure}
			
			\caption{Regions where the condition~\eqref{eq:cond2} yields existence of the discrete spectrum for $\sfD_{\tau,m}$ with interaction supported on the curve $\Sg$ as in Remark~\ref{rem:cond}.}
			\label{fig:both}
		\end{figure}


\subsection*{Acknowledgments}
VL is grateful to Graz University of Technology for the hospitality during research stays in November 2023 and Spring 2026 where a part of this work was done. MV is also grateful to Graz University of Technology for its hospitality during a research stay in March 2023.
\begin{appendix}

\section{Proof of Proposition~\ref{assumption}}
\label{app}
	Since $\Sg$ coincides with the broken line $\Sg_0$ in~\eqref{eq:Sigma0} outside a disk of a sufficiently large radius, we can first choose $s_1',s_2'\in\dR$, $s_1' < s_2'$ such that $\supp\kp \subset[s_1',s_2']$
	and that $h'=-\gamma_2(s_1') =\gamma_2(s_2') > 0$. 
	\begin{figure}[h]
\includegraphics[width=7cm, keepaspectratio]{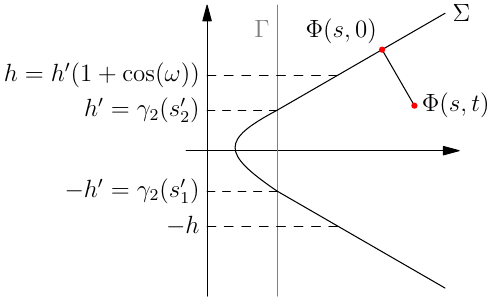}
	\caption{Schematic picture of some geometric quantities appearing in the proof of Proposition~\ref{assumption}.}
	\end{figure}
	
	Next, we fix the parameters  
	\[
	s_2 = s_2' + h'\quad\text{and}\quad s_1 = s_1' - h'.
	\]	
	Condition~(i) is clearly satisfied, because $[s_1',s_2']\subset[s_1,s_2]$.
	Condition~(ii) is also satisfied with the parameter $h$ given by the formula
	\begin{equation}\label{eq:hhprime}
	h = h'(1+\cos\omg). 
	\end{equation}
	It remains to verify that this choice satisfies condition~(iii). We will show the first identity in~(iii). The second identity can be shown analogously. Let $s > s_2$ be fixed.
	The value of the cut-radius map $c(s)$ can not be larger than the length of the line segment emanating from the point $\gamma(s)$ in the normal direction to $\Sigma$ whose second endpoint is on the bisector line $\Sigma_0$. From this observation, we conclude that
	\[
	c(s)\le \frac{h}{\sin\omg} + (s-s_2)\cot\omg.
	\]
	By definition of the cut-radius map, the equality is attained in the above inequality in the case that 
	\[
	{\rm dist}_{\dR^2}\big(\Phi(s,t),\gamma([s_1',s_2'])\big) > t \qquad \text{for all}\,\,t \in \left(0,\frac{h}{\sin\omg} + (s-s_2)\cot\omg\right),
	\]
	\ie the distance between $\Phi(s,t)$ and $\gamma(s)$ is smaller than the distance between $\Phi(s,t)$ and the arc $\gamma([s_1',s_2'])$ for all $t$ in the above interval.
	Using the convexity of the curve $\Sg$, we infer that
	\[
	{\rm dist}_{\dR^2}\big(\Phi(s,t),\gamma([s_1',s_2'])\big)
	\ge {\rm dist}_{\dR^2}\big(\Phi(s,t),\Gamma\big),
	\]
	where $\Gamma := \{(\gamma_1(s_2'),y)\in\dR^2\colon y\in\dR\}$ is the straight line that connects the points $\gamma(s_1')$ and $\gamma(s_2')$.
	Therefore, it suffices to verify that
	\begin{equation}\label{eq:toverify}
		{\rm dist}_{\dR^2}\big(\Phi(s,t),\Gamma\big) > t\qquad \text{for all}\,\,t \in \left(0,\frac{h}{\sin\omg} + (s-s_2)\cot\omg\right).
	\end{equation}
	Note that 
	\[
	\begin{aligned}
		{\rm dist}_{\dR^2}(\Phi(s,t),\G) & = 
		\gamma_1(s)-\gamma_1(s_2') + t\cos\omg \\
		&=
		(s-s_2')\sin\omg + t\cos\omg. 
	\end{aligned}	
	\]
	The inequality ${\rm dist}_{\dR^2}\big(\Phi(s,t),\Gamma\big) > t$ holds if
	\[
	t < (s-s_2+h')\frac{\sin \omg}{1-\cos \omg}.
	\]
	Thus, the condition~\eqref{eq:toverify} holds if
	\[
	\frac{h'(1+\cos \omg)}{\sin\omg} + (s-s_2)\cot\omg \le (s-s_2+h')\frac{\sin \omg}{1-\cos \omg}
	\]
	Using $\sin^2 \omega + \cos^2 \omega = 1$, the above inequality transforms to 
	\[
	(s-s_2)\left(\cot \omg-\frac{\sin \omg}{1-\cos \omg}\right) \le 0,
	\]
	which is clearly satisfied. Hence,~\eqref{eq:toverify} holds. Therefore
	\[
	c(s)\ge \frac{h}{\sin(\omega)}+(s-s_2)\cot(\omega).	
	\]
	Combining this with the opposite inequality established above, we conclude that
	\[
	c(s) =\frac{h}{\sin\omg}+ (s-s_2)\cot\omg.\qedhere
	\]

\end{appendix}


\begin{thebibliography}{999}



\bibitem{ASS}
H.~Alencar, W.~Santos, and G.~Silva Neto, 
{\it Differential geometry of plane curves},
American Mathematical Society, Providence, 2022.

\bibitem{ALTR17}
N.~Arrizabalaga, L.~Le~Treust, and N.~Raymond,
On the MIT bag model in the non-relativistic limit,
{\it Comm. Math. Phys.} {\bf 354} (2) (2017), 641--669.

\bibitem{AMV14}
N.~Arrizabalaga, A.~Mas, and L.~Vega,
\newblock Shell interactions for {D}irac operators,
\newblock {\em J. Math. Pures Appl. (9)} {\bf 102} (2014), 617--639.

	\bibitem{AMV15}
	N.~Arrizabalaga, A.~Mas, and L.~Vega,
	\newblock Shell interactions for {D}irac operators: on the point spectrum and
	the confinement,
	\newblock {\em SIAM J. Math. Anal.} {\bf 47}(2) (2015), 1044--1069.

\bibitem{BEHL17}
J. Behrndt, P. Exner, M. Holzmann, and V. Lotoreichik,
Approximation of {S}chr\"odinger operators with $\delta$-interactions supported on hypersurfaces,
{\it Math. Nachr.} {\bf 290} (2017), 1215--1248.

\bibitem{BEHL18}
J.~Behrndt, P.~Exner, M.~Holzmann, and V.~Lotoreichik,
On the spectral properties of Dirac operators with electrostatic $\dl$-shell interactions, {\it J. Math. Pures Appl. (9)} {\bf 111} (2018), 47--78.

\bibitem{BEHL19}
	J.~Behrndt, P.~Exner, M. Holzmann, and V.~Lotoreichik,
	\newblock On Dirac operators in $\R^3$ with electrostatic and Lorentz scalar $\delta$-shell interactions,
	\newblock {\em Quantum Studies} {\bf 6} (2019), 295--314.

\bibitem{BEHT25}
J. Behrndt, P. Exner, M. Holzmann, and M. Tu\v{s}ek,
On two-dimensional Dirac operators with $\delta$-shell interactions supported on unbounded curves with straight ends,
Singularities, Asymptotics, and Limiting Models, INdAM {\bf 64}, Springer (2025), 35-77.

\bibitem{BEL14}
J.~Behrndt, P.~Exner, and V.~Lotoreichik,
Schr\"{o}dinger operators with $\dl$-interactions supported on conical surfaces,
{\it J. Phys. A: Math. Theor.} {\bf 47} (2014),  355202.



\bibitem{BHOBP} J.~Behrndt, M.~Holzmann, T.~Ourmi{\`e}res-Bonafos, and K.~Pankrashkin,  Two-dimensional Dirac operators with singular interactions supported on closed curves, {\it J.~ Funct.~Anal.}~{\bf 279} (2020), 108700.

\bibitem{BHS25}
J. Behrndt, M. Holzmann, and C. Stelzer-Landauer,
Approximation of Dirac operators with $\dl$-shell potentials in the norm resolvent sense, I. Qualitative results,
{\it Math. Nachr.} {\bf 298} (2025), 2499--2546.


\bibitem{BHS26}
J. Behrndt, M. Holzmann, and C. Stelzer-Landauer,
Approximation of Dirac operators with $\delta$-shell potentials in the norm resolvent sense, II. Quantitative results, {\it Math. Nachr.} {\bf 299} (2026), 704--763.

\bibitem{BHSS24}
J.~Behrndt, M.~Holzmann, C. Stelzer-Landauer, and G. Stenzel,
\newblock Boundary triples and Weyl functions for Dirac operators with singular interactions,
\newblock {\em Rev. Math. Phys.} {\bf 36} (2024), 2350036 (65 pages).

\bibitem{BHT23}
J. Behrndt, M. Holzmann, and M. Tu\v{s}ek,
Two-dimensional Dirac operators with general $\dl$-shell interactions supported on a straight line,
{\it J. Phys. A} {\bf 56} (2023), 045201.

\bibitem{B22}
B. Benhellal,
Spectral analysis of Dirac operators with delta interactions supported on the
boundaries of rough domains,
\newblock {\em J. Math. Phys.} {\bf 63} (2022), 011507 (34 pages).

\bibitem{BBKO22}
W.~Borrelli, P.~Briet, D.~Krej\v{c}i\v{r}\'{i}k, and T.~Ourmi\`{e}res-Bonafos, 
Spectral properties of relativistic quantum waveguides,
{\it Ann. Henri Poincar\'{e}} {\bf 23} (2022), 4069--4114.

\bibitem{CEK04}
G.~Carron, P.~Exner, and D.~Krej\v{c}i\v{r}\'{i}k,
Topologically nontrivial quantum layers,
{\it J. Math. Phys.} {\bf 45} (2004),  774--784.


\bibitem{CLMT23}
B.~Cassano, V.~Lotoreichik, A.~Mas, and M.~Tu\v{s}ek, 
General $\dl$-shell interactions for the two-dimensional Dirac operator: self-adjointness and approximation,
{\it Rev. Mat. Iberoam.} {\bf 39} (2023),  1443--1492.

\bibitem{CH93}
P. Chernoff and R. Hughes,
A {N}ew {C}lass of {P}oint {I}nteractions in {O}ne {D}imension,
{\it J.~ Funct.~Anal.} {\bf 111} (1993), 97--117.

\bibitem{cycon}
H. L. Cycon, R. G. Froese, W. Kirsch, and B. Simon, \emph{Schr\"odinger operators: with applications to quantum mechanics and global geometry,} Texts and Monographs in Physics. Springer Study Edition. Springer-Verlag, Berlin, 1987

\bibitem{DLR12}
M.~Dauge, Y.~Lafranche, and N.~Raymond, 
Quantum waveguides with corners, 
{\it ESAIM, Proc.} {\bf 35} (2012),  14--45. 


\bibitem{DES89}
J.~Dittrich, P.~Exner, and P.~\v{S}eba, 
Dirac operators with a spherically symmetric $\dl$-shell interaction,
{\it J. Math. Phys.} {\bf 30} (1989),  2875--2882.

\bibitem{DE95}
P.~Duclos and P.~Exner, 
Curvature-induced bound states in quantum waveguides in two and three dimensions,
{\it Rev. Math. Phys.} {\bf 7}  (1995), 73--102.

\bibitem{DEK01}
P.~Duclos, P.~Exner, and D.~Krej\v{c}i\v{r}\'{i}k,
Bound states in curved quantum layers,
{\it Commun. Math. Phys.} {\bf 223} (2001), 13--28.

\bibitem{E20}
P.~Exner, 
Spectral properties of soft quantum waveguides,
{\it  J. Phys. A: Math. Theor.} {\bf 53} (2020),  355302.

\bibitem{E22}
P.~Exner, 
Geometrically induced spectral properties of soft quantum waveguides and layers,
{\it Rev. Math. Phys.} {\bf 36} (2022),  2360003.


\bibitem{E22-2}
P.~Exner, 
Soft quantum waveguides in three dimensions,
{\it J. Math. Phys.} {\bf 63} (2022),  042103.

\bibitem{EH22}
P.~Exner and M.~Holzmann, 
Dirac operator spectrum in tubes and layers with a zigzag-type boundary, {\it
Lett. Math. Phys.} {\bf 112} (2022), 102; correction {\bf 113} (2023), 109.


\bibitem{EI01}
P.~Exner and T.~Ichinose, 
Geometrically induced spectrum in curved leaky wires,
{\it J. Phys. A, Math. Gen.} {\bf 34}  (2001),  1439--1450.

\bibitem{EK03}
P.~Exner and S.~Kondej, 
Bound states due to a strong $\dl$-interaction supported by a curved surface, 
{\it J. Phys. A, Math. Gen.} {\bf 36}  (2003), 443--457.

\bibitem{EKL18}
P.~Exner, S.~Kondej, and V.~Lotoreichik, 
Asymptotics of the bound state induced by $\dl$-interaction supported on a weakly deformed plane, {\it J. Math. Phys.} {\bf 59} (2018),  013501.

\bibitem{EKL24}
P.~Exner, S.~Kondej, and V.~Lotoreichik, 
Bound states of weakly deformed soft waveguides, 
{\it Asymptot. Anal.} {\bf 138} (2024),  151--174.

\bibitem{EK15}
P.~Exner and H.~Kova\v{r}\'{i}k, 
{\it Quantum waveguides,}
Springer, Cham, 2015.

\bibitem{ES89}
P.~Exner and P.~\v{S}eba, 
Bound states in curved quantum waveguides,
{\it J. Math. Phys.} {\bf 30} (1989),  2574--2580.

\bibitem{ES24}
P.~Exner and D.~Spitzkopf, 
Tunneling in soft waveguides: closing a book,
{\it  J. Phys. A: Math. Theor.} {\bf 57} (2024),  125301.


\bibitem{FHL24}
D. Frymark, M. Holzmann, and V. Lotoreichik,
Spectral analysis of the Dirac operator with a singular interaction on a broken line,
{\it J. Math. Phys.} {\bf 65} (8) (2024), 083514 (24 pages).

\bibitem{FL23}
D.~Frymark and V.~Lotoreichik,
Self-adjointness of the 2D Dirac operator with singular interactions supported on star graphs, {\it Ann. Henri Poincar\'{e}} {\bf 24} (2023), 179--221.


\bibitem{H64}
P.~Hartman, Geodesic parallel coordinates in the large, 
{\it Amer. J. Math.} {\bf 86} (1964), 705--727.

\bibitem{H26}
M. Holzmann,
Approximation of magnetic {S}chr\"odinger operators with $\delta$-interactions supported on networks,
{\it Lett. Math. Phys.} {\bf 116} (2026), 24 (25 pages).

\bibitem{HOP18}
M.~Holzmann, T.~Ourmi{\`e}res-Bonafos, and K.~Pankrashkin,
\newblock {D}irac operators with {L}orentz scalar shell interactions,
\newblock {\it Rev. Math. Phys.} {\bf 30} (2018), 1850013 (46 pages).

\bibitem{khalile}
M. Khalile and K. Pankrashkin, Eigenvalues of Robin Laplacians in infinite sectors,
{\it  Math. Nachr.} {\bf 291} (2018), 928--965.

\bibitem{KKK21}
S.~Kondej, D.~Krej\v{c}i\v{r}\'{i}k, and J.~K\v{r}\'{i}\v{z}, 
Soft quantum waveguides with an explicit cut-locus,
{\it J. Phys. A, Math. Theor.} {\bf 54} (2021),  30LT01.


\bibitem{KK24}
D.~Krej\v{c}i\v{r}\'{i}k and J.~K\v{r}\'{i}\v{z}, 
Bound states in soft quantum layers,
{\it Publ. Res. Inst. Math. Sci.} {\bf 60} (2024),  741--766.

\bibitem{LO16}
V.~Lotoreichik and T.~Ourmi\`{e}res-Bonafos, 
On the bound states of Schr\"odinger operators with $\dl$-interactions on conical surfaces, {\it Commun. Partial Differ. Equations} {\bf 41} (2016),  999--1028.

\bibitem{LO18}
V.~Lotoreichik and T.~Ourmi{\`e}res-Bonafos,
A Sharp Upper Bound on the Spectral Gap for Graphene Quantum Dots,
{\it Math. Phys. Anal. Geom.} \textbf{22} (2019), 13.


\bibitem{MP18}
A.~Mas and F.~Pizzichillo, 
Klein's paradox and the relativistic $\dl$-shell interaction in $\dR^3$,
{\it Anal. PDE} {\bf 11} (2018),  705--744.

\bibitem{M00}
W.~McLean,
\newblock {\it Strongly Elliptic Systems and Boundary Integral Equations},
\newblock Cambridge University Press, Cambridge, 2000.


\bibitem{MT}
A.~Morame and F.~Truc, Remarks on the spectrum of the Neumann problem with magnetic field in the half-space,
{\it J. Math. Phys.} {\bf 46} (2005) 1-13.

\bibitem{NGM04}
K.S.~Novoselov, A.K.~Geim, S.V.~Morozov, D.~Jiang, Y.~Zhang, S.V.~Dubonos, I.V.~Grigorieva, and A.A.~Firsov, Electric field effect in atomically thin carbon films, \emph{Science} \textbf{306} (2004), 666--669.


\bibitem{OP18}
T.~Ourmi\`{e}res-Bonafos and K.~Pankrashkin,
Discrete spectrum of interactions concentrated near conical surfaces, {\it 
Appl. Anal.} {\bf 97} (2018),  1628--1649.


\bibitem{R21}
V.~Rabinovich,
Two-Dimensional Dirac Operators with Interactions on Unbounded Smooth Curves,
{\em Russ. J. Math. Phys.} {\bf 28} (2021), 524--542.


\bibitem{R22a}
V.~Rabinovich,
Dirac Operators with Delta-Interactions on Smooth Hypersurfaces in $\mathbb{R}^n$,
{\em J. Fourier Anal. Appl.} {\bf 28} (2022), 20.


\bibitem{RSIV}
M.~Reed and B.~Simon, {\it 
Methods of modern mathematical physics. IV: Analysis of operators},
Academic Press, 1978.

\bibitem{S01}
A.~Savo, 
Lower bounds for the nodal length of eigenfunctions of the Laplacian, {\it 
Ann. Global Anal. Geom.} {\bf 19}  (2001),  133--151.


\bibitem{SST03}
K.~Shiohama, T.~Shioya, and M.~Tanaka, 
{\it The geometry of total curvature on complete open surfaces},
Cambridge University Press, 2003.

\bibitem{simon}
B. Simon, On the number of bound states of two body Schr\"odinger operators - A review, {\it Studies in Mathematical Physics: Essays in Honor of Valentine Bargman} (1976), 305--326.

\bibitem{T92}
B.~Thaller,
{\it The {D}irac {E}quation},
Texts and Monographs in Physics, Springer-Verlag, 1992.


\end{thebibliography}
\end{document}